\documentclass[reqno, 12pt]{amsart}

\usepackage{amsmath,graphicx}

\usepackage[algoruled,lined,noresetcount,norelsize]{algorithm2e}

\usepackage{amssymb,amsthm}
\usepackage{mathrsfs}
\usepackage{mathabx}\changenotsign
\usepackage{dsfont}
\usepackage{xcolor}
\usepackage[backref]{hyperref}
\hypersetup{
    colorlinks,
    linkcolor={red!60!black},
    citecolor={green!60!black},
    urlcolor={blue!60!black}
}
\usepackage[T1]{fontenc}
\usepackage{lmodern}
\usepackage[babel]{microtype}
\usepackage[english]{babel}

\usepackage{geometry}
\numberwithin{equation}{section}
\numberwithin{figure}{section}

\usepackage{enumitem}

\theoremstyle{plain}
\newtheorem{thm}{Theorem}[section]

\newtheorem{obs}[thm]{Observation}

\newtheorem{lemma}[thm]{Lemma}

\newtheorem{conj}[thm]{Conjecture}
\newtheorem{question}[thm]{Question}
\newtheorem{problem}[thm]{Problem}

\theoremstyle{definition}

\newcommand\dist{\text{dist}}

\renewcommand{\epsilon}{\varepsilon}

\def\tref#1{ {\tiny \ref{#1} } }

\def\itm#1{\rm ({#1})}

\def\itmarab#1{\mbox{\itm{{\it #1\,}\arabic{*}\hspace{.05em}}}}

\renewcommand{\epsilon}{\varepsilon}

\begin{document}

\title[Fast Waiter-Client games]{Fast strategies in Waiter-Client games on $K_n$}
\author[Dennis Clemens]{Dennis Clemens}
\address{(DC) Technische Universit\"at Hamburg, Institut f\"ur Mathematik, Am Schwarzenberg-Campus 3, 21073 Hamburg, Germany }
\email{dennis.clemens@tuhh.de}
\author[Pranshu Gupta]{Pranshu Gupta}
\address{(PG) Technische Universit\"at Hamburg, Institut f\"ur Mathematik, Am Schwarzenberg-Campus 3, 21073 Hamburg, Germany }
\email{pranshu.gupta@tuhh.de}
\author[Fabian Hamann]{Fabian Hamann}
\address{(FH) Technische Universit\"at Hamburg, Institut f\"ur Mathematik, Am Schwarzenberg-Campus 3, 21073 Hamburg, Germany }
\email{fabian.hamann@tuhh.de}
\author[Alexander Haupt]{Alexander Haupt}
\address{(AH) Technische Universit\"at Hamburg, Institut f\"ur Mathematik, Am Schwarzenberg-Campus 3, 21073 Hamburg, Germany }
\email{alexander.haupt@tuhh.de}
\author[Mirjana Mikala\v{c}ki]{Mirjana Mikala\v{c}ki}
\address{(MM) University of Novi Sad, Faculty of Sciences, Department of Mathematics and Informatics, Trg Dositeja Obradovi\v{c}a 4, 21000 Novi Sad, Serbia}
\email{mirjana.mikalacki@dmi.uns.ac.rs}
\author[Yannick Mogge]{Yannick Mogge}
\address{(YM) Technische Universit\"at Hamburg, Institut f\"ur Mathematik, Am Schwarzenberg-Campus 3, 21073 Hamburg, Germany }
\email{yannick.mogge@tuhh.de}

\pagestyle{plain}

\date{\today}

\maketitle

% Abstract
\begin{abstract}
Waiter-Client games are played on some hypergraph $(X,\mathcal{F})$, where $\mathcal{F}$ denotes the family of winning sets. For some bias $b$, during each round of  
the game Waiter offers $b+1$ elements of $X$ to Client from which he claims one for himself and the rest go to Waiter. Proceeding like this Waiter wins the game if she forces Client to claim all the elements of any winning set from $\mathcal{F}$. In this paper we study fast strategies for several Waiter-Client games played on the edge set of the complete graph, i.e.~$X=E(K_n)$, in which the winning sets are perfect matchings, Hamilton cycles,
pancyclic graphs, fixed spanning trees or factors of a given graph.
\end{abstract}

% Introduction
\section{Introduction}

\subsection{Positional games.}
Positional games belong to the family of perfect information games between two players, 
and they have become a field of intense studies throughout the last decades 
(for an introductory overview see e.g.~\cite{HKSSbook} \cite{Kriv2014}). 
Starting with the seminal papers by Hales and Jewett~\cite{HJ1963} 
as well as Erd\H{o}s and Selfridge~\cite{ES1973}, 
many beautiful results have been proven since then.
Some of these even provide intriguing connections between positional games and 
other branches of combinatorics such as
random graph theory (see e.g. \cite{Beckbook}, \cite{BHKL2016}, \cite{BL2000}, 
\cite{FKN2015}, \cite{Kriv2011}, \cite{KT2018}), 
extremal combinatorics and Ramsey theory 
(see e.g. \cite{Beckbook}, \cite{HJ1963}, \cite{NSS2016}).  

Related to the so-called {\em strong games}, 
among which Tic-Tac-Toe and Hex are probably the most famous examples (see e.g.~\cite{Beckbook}),
{\em Maker-Breaker} games were introduced in~\cite{ES1973} and have since become the most studied type of positional games.
In a more general form, the latter are played as follows. 
Given a hypergraph $(X,\mathcal{F})$ and a positive integer $b$,  two players, Maker and Breaker, alternate in turns. In each round Maker is allowed to claim (up to) one element from the {\em board X}, which has not been claimed before in the game, while Breaker is allowed to claim (up to) $b$ such elements in each round. 
If, until the end of the game, Maker succeeds in fully claiming all the elements of any {\em winning set} $F\in\mathcal{F}$,
she wins the game. Otherwise, Breaker has claimed at least one element of each of the winning sets, and is declared the winner of the game. No draw is possible. Here, the integer $b$ is called the {\em bias} of Breaker. 
In the case when $b=1$ holds, the game is called {\em unbiased},
and in all other cases the game is called $b$-biased or {\em biased} for short. 
Note that it also makes sense to introduce a bias for Maker (see e.g.~\cite{CM2018}, \cite{HMS2012}),
but we will not consider this case in the remainder of our paper.

Many natural games played on the edge set of a complete graph $K_n$ (i.e.~ $X=E(K_n)$) 
are easy wins for Maker in the unbiased setting. As an example, consider the {\em perfect matching game}. 
Here, Maker aims at claiming all the edges of a perfect matching (of $K_n$), and it is a rather easy exercise to see that Maker has a strategy to win this game, provided $n$ is sufficiently large. 
Having a simple winning strategy raises a few natural questions:
What happens if we give Breaker more power by increasing his bias $b$?
-- This then leads to the question of finding the {\em threshold bias} 
around which a Maker's win suddenly turns into a Breaker's win (see e.g.~\cite{GS2009}, \cite{Kriv2011}).
What happens if we restrict Maker's options by playing on a sparser
board? -- A typical approach to this question is to play the game on the edges of a sparse random graph $G_{n,p}$,
and to ask, around which value of $p$ a likely Maker's win
turns into a likely Breaker's win (see e.g.~\cite{NSS2016}, \cite{SS2005}). Also, what can we say if we want Maker to win as fast as possible?

The last question is interesting in its own right, since
an answer to it can be seen as a measure for how powerful
Maker is. But finding fast winning strategies
is even more relevant as sometimes winning strategies for more 
involved games can be given by splitting the game into several stages, in which Maker first aims to create a simpler, or a nice behaving, structure as fast as possible, so that afterwards she still has enough options to extend 
this structure into a full winning set (for an example, see e.g.~\cite{Kriv2011}).
In our paper, we will focus on finding the fast winning strategies.
However, we will consider {\em Waiter-Client games} instead of Maker-Breaker games.

\subsection{Waiter-Client games}

Waiter-Client games and Client-Waiter games, 
formerly known as Chooser-Picker games and Picker-Chooser games (see e.g.~\cite{Beckbook}),
have received increasing attention recently. Again, played on some hypergraph
$(X,\mathcal{F})$ with some bias $b$, the players claim elements from $X$
and exactly one player aims for a winning set while the other one wants to prevent that.
This time, instead of claiming the elements of $X$ alternately,
the distribution of the board elements is done by the following rule:
In every round, Waiter picks $b+1$ previously unclaimed elements of the board and Client chooses exactly one of these elements to be claimed by him,
while all the other elements go to Waiter's graph. 
In the Waiter-Client game, Waiter is said to be the winner if she makes
Client claim all the elements 
of any winning set from $\mathcal{F}$. 
In this case we say that Waiter
{\em forces} Client to occupy a winning set. 
In the Client-Waiter game, Waiter however is said to be the winner if she prevents 
Client from claiming all the elements of any winning set until the end of the game. Otherwise, Client is declared the winner.

One may wonder whether there is some connection between the above games and Maker-Breaker games in general.
Beck~\cite{Beckbook} observed that for a few natural 
families of winning sets $\mathcal{F}$
Waiter wins the Waiter-Client (or Client-Waiter) game more 
easily than Maker (or Breaker) wins the corresponding Maker-Breaker game. Later on, this was conjectured by 
Csernenszky, M\'andity and A.~Pluh\'ar~\cite{CMP2009}.

\begin{conj}[Conjecture 1 in~\cite{CMP2009}]
Waiter wins a Waiter-Client (Client-Waiter) game on 
$(X,\mathcal{F})$ if Maker (Breaker) as second player 
wins the corresponding Maker-Breaker game.
\end{conj}

While the above conjecture has been disproved
by Knox~\cite{Knox2012} recently, there is still a chance that Beck's intuition holds for many typical winning families $\mathcal{F}$. A few examples of the games supporting the conjecture are already given by M.~Bednarska-Bzd\c{e}ga in~\cite{BB2013}. In our paper we will provide a variety of examples in which
Waiter in a Waiter-Client game wins at least as fast
as Maker does in the corresponding Maker-Breaker game. To make this more precise, let 
$$
\tau_{MB}(\mathcal{F},b) ~~ \text{and} ~~ \tau_{WC}(\mathcal{F},b)
$$
denote the smallest integer $t$ such that Maker or Waiter, respectively, has a strategy to win the $b$-biased game
with winning family $\mathcal{F}$ within $t$ rounds.

Observe first that both values are bounded trivially from below by $m_{\mathcal{F}}:=\min\{|F|:~ F\in\mathcal{F}\}$. In case that Waiter (or Maker) can win in exactly this number of rounds, we will say that the game is won {\em perfectly fast},
while we say that the game is won {\em asymptotically fast}
when the number of rounds is of size $(1+o(1))m_{\mathcal{F}}$.

\medskip

{\bf Perfect matching game.} As stated earlier, the perfect matching game is an easy win for Maker in the unbiased 
Maker-Breaker game on $E(K_n)$. To be more precise, let us denote with $\mathcal{PM}_n$
the family of all perfect matchings of $K_n$. Hefetz, Krivelevich, Stojakovi\'c and Szab\'o~\cite{HKSS2009}
proved that $\tau_{MB}(\mathcal{PM}_n,1)=\frac{n}{2}+1$ for every large enough even $n$, thus showing that the unbiased perfect matching game is won asymptotically fast.
In fact, at the moment when Maker wins the game, her graph consists of a perfect matching and
at most one {\em wasted} edge. We will show that the same is true
for the unbiased Waiter-Client game on $K_n$.

\begin{thm}\label{thm:pm_unbiased}
For every large enough even integer $n$ the following holds:
$$
\tau_{WC}(\mathcal{PM}_n,1)=\frac{n}{2}+1~ .
$$
\end{thm}

Moreover, for the biased perfect matching game the fifth author and Stojakovi\'c~\cite{MS2016} proved that %even 
for 
$b=O\left(\frac{n}{\ln n}\right)$ the perfect matching game
is won asymptotically fast. Note that the bound on $b$
is best possible (in its order of magnitude), as for $b>(1+o(1))\frac{n}{\ln n}$ Breaker has a strategy for isolating a vertex 
in Maker's graph~\cite{CE1978}.

\begin{thm}[Theorem 1.1 and Theorem 1.5(i) in \cite{MS2016}]\label{thm:pm_mima}
There exist constants $\delta,c,C>0$ such that for every large enough $n$ and for every $b\leq \delta \frac{n}{\ln n}$
the following holds.
$$
\frac{n}{2} + c b \leq \tau_{MB}(\mathcal{PM}_n,b)\leq \frac{n}{2} + C b \ln b~ .
$$
\end{thm}

Note that in the above statement the second order term is determined up to a logarithmic factor.
For the Waiter-Client version of that game, we can get rid of this extra factor in the upper bound
while also allowing the bias to be linear in $n$.

\begin{thm}\label{thm:pm_biased}
There exist constants $\delta,C>0$ such that for every large enough even $n$ and for every $b\leq \delta n$
the following holds.
$$
\tau_{WC}(\mathcal{PM}_n,b)\leq \frac{n}{2}+ C b~ .
$$
\end{thm}

\medskip

{\bf Hamiltonicity game.} Another game which is easily won by Maker when played on the edges of 
a complete graph $K_n$ is the Hamiltonicity game. This time, 
let $\mathcal{H}_n$ denote the family of all Hamilton cycles of $K_n$,
and for a moment consider $n$ to be sufficiently large. 
Already in their early paper, Chv\'atal and Erd\H{o}s~\cite{CE1978}
could show that Maker wins this game in at most $2n$ rounds.
This was later improved to $\tau_{MB}(\mathcal{H}_n,1)\leq n+2$
by Hefetz, Krivelevich, Stojakovi\'c and Szab\'o~\cite{HKSS2009}
through ad-hoc arguments coupled with the P\'osa rotation technique~\cite{Posa1976}. Finally, Hefetz and Stich~\cite{HS2009} 
were fighting for the exact result and
proved that
$\tau_{MB}(\mathcal{H}_n,1)= n+1$ 
by providing a rather technical (13 pages long) proof involving multiple case distinctions.
We will show that Waiter can win the Waiter-Client version of
the unbiased Hamiltonicity game in the same number of rounds.

\begin{thm}\label{thm:ham_unbiased}
For every large enough integer $n$ the following holds:
$$
\tau_{WC}(\mathcal{H}_n,1)=n+1~ .
$$
\end{thm}

Considering the biased Hamiltonicity game, the following has been shown.

\begin{thm}[Theorem 1.3 and Theorem 1.5(ii) in \cite{MS2016}]\label{thm:ham_mima}
There exist constants $\delta,c,C>0$ such that for every large enough $n$ and for every $b\leq \delta \left(\frac{n}{\ln^5 n}\right)^{\frac{1}{2}}$
the following holds.
$$
n + c b \leq \tau_{MB}(\mathcal{H}_n,b)\leq n + C b^2 \ln^5 b~ .
$$
\end{thm}

For the Waiter-Client version of that game, we are able to show a better upper bound on the number of rounds,
while again allowing the bias to be of linear size.

\begin{thm}\label{thm:ham_biased}
There exist constants $\delta,C>0$ such that for every large enough $n$ and for every $b\leq \delta n$
the following holds.
$$
\tau_{WC}(\mathcal{H}_n,b)\leq n + C b~ .
$$
\end{thm}

\medskip

{\bf Pancyclicity game.} 
Quite recently (in \cite{BHKL2016} and \cite{FKN2015}) it was suggested to
generalise the Hamiltonicity game even further by choosing the winning sets to be all the
{\em pancyclic} subgraphs of $K_n$ -- subgraphs containing the cycles of all possible lengths between $3$ and $n$.
Denote with $\mathcal{PC}_n$ the family of all such subgraphs.
Ferber, Krivelevich and Naves~\cite{FKN2015} proved that for $b=o(\sqrt{n})$ the $b$-biased Maker-Breaker pancyclicity game on $K_n$ is won by Maker,
while it was already known before that for $b\geq 2\sqrt{n}$ Breaker wins, since he can block all triangles~\cite{CE1978}.
In contrast to this result, it was shown by Bednarska-Bzd\c{e}ga, Hefetz, Krivelevich and \L uczak~\cite{BHKL2016} that the threshold bias in the corresponding Waiter-Client game is linear in $n$.
Apart from that, not so much is known for games with $\mathcal{PC}_n$ being the family of winning sets. 
In particular, no tight results on the number of rounds has been given before. In this paper we prove the following.

\begin{thm}\label{thm:pancyclic_fast}
In the unbiased Waiter-Client pancyclicity game the following holds: 
$$
\tau_{WC}(\mathcal{PC}_n,1) = n + (1+o(1))\log_2 n.
$$
\end{thm}

Note that this means that Waiter wins almost perfectly fast,
as every spanning pancyclic subgraph of $K_n$
has at least $n+(1-o(1))\log_2 n$ edges~\cite{Bondy1971}. Moreover, the second order term in the above theorem will be made even more precise later on (see the remark at the end of Section~\ref{sec:pancyclic}).

\medskip

{\bf Connectivity and fixed spanning trees.} Another game easily won by Maker in its unbiased version
is the so-called {\em connectivity game} on $K_n$, introduced by Chvat\'al and Erd\H{o}s~\cite{CE1978} in which Maker's goal is to claim any spanning subgraph of $K_n$. Indeed, we already discussed that
for large enough $n$, Maker has a strategy to create a Hamilton cycle asymptotically fast. Moreover,
since there is no reason for Maker to close cycles in the connectivity game, there needs to be a strategy which succeeds within
$n-1$ rounds. In fact, following the result of Lehman~\cite{Lehm1964} Maker can %even 
win the game
when $K_n$ is replaced by any graph consisting of two edge-disjoint spanning trees.

Due to the simplicity of the aforementioned game, Ferber, Hefetz and Krivelevich~\cite{FHK2012} introduced
a variant of the connectivity game in which Maker aims to occupy a copy of some given spanning tree $T$.
Obviously, in order to have a winning strategy for Maker, 
we %then 
cannot choose $T$ arbitrarily now 
as Breaker can block large stars. Thus, it is natural to put some degree constraints
on the desired tree $T$. Let $\mathcal{F}_T$ denote the family of all copies of $T$ in $K_n$.
The following result has been proven first.

\begin{thm}[Theorem 1.1 in \cite{FHK2012}]
Let $\varepsilon>0$. Then for every large enough integer $n$ the following holds. Let $T$ be any spanning tree
on $n$ vertices, with maximum degree $\Delta(T)\leq n^{0.05-\varepsilon}$, and let $b\leq n^{0.005-\varepsilon}$
be any positive integer. Then 
$$
\tau_{MB}(\mathcal{F}_T,b)=n+o(n)~ .
$$
\end{thm}

Thus, even when the maximum degree and Breaker's bias are increasing with $n$, Maker has a strategy to win the fixed spanning tree game asymptotically fast. Naturally, one may wonder whether the error term in the above theorem can be improved
when we put stronger constraints on $\Delta(T)$ and $b$. This was answered in~\cite{CFGHL2015}
as follows.

\begin{thm}[Theorem 1.1 and Theorem 1.4 in~\cite{CFGHL2015}]
Let $\Delta$ be a positive integer, then for every large enough integer $n$ the following holds.
For every spanning tree $T$ on $n$ vertices with maximum degree $\Delta(T)\leq \Delta$,
we have
$$
n-1 \leq \tau_{MB}(\mathcal{F}_T,1) \leq n+1~ .
$$
Moreover, if $T$ is a tree chosen uniformly at random among all labeled trees on $n$ vertices (not necessarily
having a constant bound on the maximum degree), then with high probability
$$
\tau_{MB}(\mathcal{F}_T,1) = n-1~ .
$$
That is, for most choices of $T$, Maker wins the tree embedding game perfectly fast.
\end{thm}

In our paper we will show that in the unbiased Waiter-Client game, Waiter has a fast winning strategy
in the game $(E(K_n),\mathcal{F}_T)$ which creates at most one wasted edge.
Moreover, in contrast to the above theorems we may also allow the maximum degree to grow much faster with $n$.

\begin{thm}\label{thm:tree}
There exists a constant $\varepsilon>0$ such that the following holds for every large enough integer $n$.
Let $T$ be spanning tree of $K_n$ with $\Delta(T)\leq \varepsilon \sqrt{n}$, then
$$
n-1 \leq \tau_{WC}(\mathcal{F}_T,1) \leq n .
$$
Moreover, the lower and the upper bounds are tight, surprisingly also for $\Delta(T)\leq 3$.
\end{thm}

{\bf $H$-factor game.} Using our methods from the fixed spanning tree game, we are also able to describe
fast winning strategies for games in which Waiter aims to create a factor of a fixed constant size tree.
Note that the same kind of question was studied in the 
Maker-Breaker setting, but
only in the case when the fixed tree is either a path or a star~\cite{CM2018}. 

More precisely, for a fixed graph $H$ and an integer $n$ satisfying $v(H)|n$, 
an {\em $H$-factor} of $K_n$ is defined to be the vertex disjoint union of 
copies of $H$ covering all vertices of $K_n$. Let $\mathcal{F}_{n,H-fac}$ be the family of all such 
subgraphs. We prove the following result:

\begin{thm}\label{thm:tree_factor}
Let $k\geq 2$ be a positive integer and let $T$ be any fixed tree on $k$ vertices.
Provided that $n$ is a large enough integer with $k|n$, the following holds: 
$$
\frac{k-1}{k}n \leq \tau_{WC}(\mathcal{F}_{n,T-fac},1) \leq \frac{k-1}{k}n + 1 .
$$
Moreover, the lower and the upper bound are tight.
\end{thm}

Observe that in all the games considered so far, Waiter can always win at least asymptotically fast.
We finish the paper by giving an example where this is not the case and challenge the reader to
improve our bounds. 
We consider the {\em triangle factor game}, whose Maker-Breaker version has been discussed in~\cite{ABKNP2017} and \cite{FHK2012}.

\begin{thm}\label{thm:triangle_factor}
For every large enough integer $n$ such that $3|n$ the following holds:
$$
\frac{13}{12}n \leq \tau_{WC}(\mathcal{F}_{n,K_3-fac},1) \leq \frac{7}{6}n + o(n) .
$$
\end{thm}

We note that the above game again provides an example where Waiter can win
(asymptotically) at least as fast as Maker in the corresponding
Maker-Breaker version. Indeed, as was observed by Szab\'o (a proof is contained in~\cite{FHK2012}),
Maker cannot win the unbiased triangle factor game on $K_n$ within less than $\frac{7}{6}n$ rounds.

\medskip\medskip

{\bf Organisation of the paper.}
In Section~\ref{sec:pm} and Section~\ref{sec:ham} we focus on Theorem~\ref{thm:pm_unbiased} and Theorem~\ref{thm:ham_unbiased}, proving slightly stronger statements
so that we can use those as tools for the proofs in later sections. In Section~\ref{sec:pancyclic} we %then 
prove Theorem~\ref{thm:pancyclic_fast}, in Section~\ref{sec:trees}
we show Theorem~\ref{thm:tree} and Theorem~\ref{thm:tree_factor}, in Section~\ref{sec:triangles}
we continue with the proof of Theorem~\ref{thm:triangle_factor}, and in Section~\ref{sec:biased}
we prove Theorem~\ref{thm:pm_biased} as well as Theorem~\ref{thm:ham_biased}. Whenever we use results or methods from other papers, we will introduce the necessary concepts
in the relevant sections. Finally, we finish the paper with a few concluding remarks and open problems in Section~\ref{sec:concluding}.

\medskip

{\bf Notation.}
The graph-theoretic notation that we use is standard 
and closely follows the notation from~\cite{W2001}.
We write $[n]:=\{k\in\mathbb{N}:~ 1\leq k\leq n\}$ for every positive integer $n$. 

Given any graph $G$, we let $V(G)$ and $E(G)$
denote the vertex set and the edge set of $G$, respectively,
and set $v(G)=|V(G)|$ and $e(G)=|E(G)|$.
If $\{v,w\}\in E(G)$ is an edge, %then 
we denote it with $vw$ for short,
and we call $w$ a neighbour of $v$. We set
$N_G(v)=\{w\in V(G):~ vw\in E(G)\}$ to be the neighbourhood of $v$ in $G$ and call $d_G(v)=|N_G(v)|$ the degree of $v$ in $G$.
Moreover, $\Delta(G)=\max_{v\in V(G)} d_G(v)$
denotes the maximum degree of $G$ 
and $\delta(G) = \min_{v\in V(G)} d_G(v)$ denotes the minimum degree of $G$.
Given any two subsets $A,B\subset V(G)$ and any vertex $v\in V(G)$
we write 
$N_G(v,A)=N_G(v)\cap A$,  
$d_G(v,A)=|N_G(v,A)|$,
$N_G(A):=\bigcup_{v\in A} N_G(v)$, 
$E_G(A):=\{vw\in E(G):~ v,w\in A\}$,
$e_G(A):=|E_G(A)|$,
$E_G(A,B):=\{vw\in E(G):~ v\in A,w\in B\}$
and 
$e_G(A,B)=|E_G(A,B)|$.

For any two graphs $H$ and $G$ we write $H\subset G$
if both $V(H)\subset V(G)$ and $E(H)\subset E(G)$ hold,
and call $H$ a subgraph of $G$ in this case.
We also set $G\setminus H=(V(G),E(G)\setminus E(H))$ in this case.
Given any $A\subset V(G)$, we call $G[A]=(A,E_G(A))$ 
the subgraph of $G$ induced by $A$.

Two graphs $H$ and $G$ are called isomorphic, denoted with $H\cong G$,
if there exists a bijection $f:V(H)\rightarrow V(G)$
such that $vw$ is an edge of $H$ if and only if
$f(v)f(w)$ is an edge of $G$. If the latter is the case,
we also say that $H$ forms a copy of $G$. 

If we represent a path $P$ by a sequence $(v_1,v_2,\ldots,v_k)$,
this means that $V(P)=\{v_1,v_2,\ldots,v_k\}$ and
$E(P)=\{v_iv_{i+1}:~ 1\leq i\leq k-1\}$ hold. 
Similarly, if a cycle $C$ is represented by
a sequence $(v_1,v_2,\ldots,v_k)$,
we mean that $V(C)=\{v_1,v_2,\ldots,v_k\}$ and
$E(C)=\{v_iv_{i+1}:~ 1\leq i\leq k-1\}\cup \{v_kv_1\}$. 
The length of both a path and a cycle is always the number of its edges.

Assume that some Waiter-Client game is in progress. 
We let $W$ and $C$ denote the graphs consisting of Waiter's edges and 
Client's edges, respectively.
Any edge belonging to $C\cup W$ is said to be claimed,
while all the other edges in play are called free.

% Unbiased perfect matching game
\section{Unbiased perfect matching game}\label{sec:pm}

In the section we will prove Theorem~\ref{thm:pm_unbiased} 
by showing a slightly stronger statement which will also be applied later in the discussion of the tree embedding game (Section~\ref{sec:trees}).

\begin{thm}\label{WCUPM}
For large enough $n$, the following holds: Let $H\subset K_{n,n}$ be any subgraph with $e(H)\leq \frac{n}{2}$. Then, in the unbiased Waiter-Client game on $K_{n,n}\setminus H$, Waiter has a strategy to force a perfect matching of 
$K_{n,n}$ within $n+1$ rounds.
\end{thm}

\begin{proof}
Let $V=A\cup B$ be the bipartition of $K_{n,n}$. Throughout the game, we 
denote with $R$ the set of vertices which are isolated in Client's graph. 
For $n-4$ rounds (Stage~I), Waiter's strategy will be to force a large matching in Client's graph greedily, while making sure that $e_{W\cup H}(R)$ decreases with every round as long as this value is still positive. Within 5 further rounds (Stage~II), Waiter will then complete this matching
to a perfect matching.

If at any point during the game, Waiter is unable to follow her strategy, she forfeits the game. (We will later see that this does not happen.) The set
$R$ is dynamically updated after every turn. Waiter's strategy consists of the following two stages: \\

\textbf{Stage I:} This stage lasts $n-4$ rounds, in which Waiter forces
a matching of size $n-4$ in Client's graph between $A$ and $B$. 
Each round is played as follows: Let $u\in A\cap R$ be an arbitrary vertex maximizing $d_{W\cup H}(u,B\cap R)$. Then Waiter offers two free edges $ub_1$, $ub_2$ with $b_1,b_2\in B\cap R$. By symmetry, assume that 
Client claims $ub_1$. Then, vertices $u$ and $b_1$ are removed from $R$.

\medskip

\textbf{Stage II:} When Waiter enters Stage~II, Client's graph is a matching $M'$ of size $n-4$. Let $S=A\setminus V(M')$ and $T=B\setminus V(M')$ at this point. 
Within $5$ rounds, Waiter forces a matching of size $4$
between $S$ and $T$. The details are given later in the strategy discussion.

\medskip

It is evident that, if Waiter can follow this strategy without forfeiting the game, she creates a perfect matching of $K_{n,n}\setminus H$ in Client's graph within $n+1$ rounds. It thus remains to check that she does not forfeit the game.\\

\textbf{Strategy Discussion:}

{\bf Stage I:} By induction on the number of rounds
it follows that Waiter can follow the strategy of Stage~I
while she additionally maintains that
\begin{equation}\label{eq:pm}
e_{W\cup H}(R) \leq \max\left\{0, \frac{|R|-n}{2} \right\}~ .
\end{equation}
Indeed, the above inequality holds at the beginning of the game, since
at that point $|R|=2n$ and $e(H\cup W)\leq \frac{n}{2}$.
Now consider any round $r$ in Stage~I, and assume that Waiter so far
could follow the strategy and maintain Inequality~\ref{eq:pm}.
According to the strategy, she then picks a vertex
$u\in A\cap R$ such that $d_{W\cup H}(u,B\cap R)$ is maximal.
By induction we have that the number of vertices $b\in B\cap R$
with $ub$ being free is at least
$$
|B\cap R| - d_{W\cup H}(u,B\cap R) 
\geq \frac{|R|}{2} - \max\left\{0, \frac{|R|-n}{2} \right\}
= \min\left\{\frac{|R|}{2}, \frac{n}{2} \right\} \geq 2~ .
$$
Hence, there exist at least two vertices $b_1,b_2\in B\cap R$
as required and Waiter can follow her strategy.
Moreover, if $e_{W\cup H}(R)=0$ was true at the beginning of round $r$,
then the same still holds after the update of round $r$,
since $u$ gets removed from $R$. Otherwise, $u$ was chosen
with $d_{W\cup H}(u,B\cap R)\geq 1$ and, after the update, $|R|$ decrases by $2$ while $e_{W\cup H}(R)$ decreases by at least  
$d_{W\cup H}(u,B\cap R)\geq 1$. In any case, Inequality~\ref{eq:pm}
holds again.

\medskip

{\bf Stage~II:} When Waiter enters Stage~II, Client's graph is a matching $M'$ of size $n-4$. Moreover, using Inequality~\ref{eq:pm} from the end of Stage~I, 
we deduce that $e_{W\cup H}(S,T) = 0$ needs to
hold at that point. % where $S=A\setminus V(M')$ and $T=B\setminus V(M')$.
Waiter now forces a perfect matching between $S=\{s_1,s_2,s_3,s_4\}$ and
$T=\{t_1,t_2,t_3,t_4\}$ as follows: In the first two rounds, she offers all the edges $s_1t_i$ with $i\in [4]$. W.l.o.g.~we may assume that Client chooses
$s_1t_1$ and $s_1t_2$. Then, in the third round, Waiter offers $s_2t_3$ and $s_2t_4$, from which Client w.l.o.g.~chooses $s_2t_3$. Afterwards,
Waiter offers the edges $s_3t_4$ and $s_4t_4$, from which Client w.l.o.g.~chooses $s_3t_4$. Finally, Waiter offers $s_4t_1$ and $s_4t_2$, and no matter which edge Client takes, a perfect matching is now finished. %then.
\end{proof}

\textit{Proof of Theorem~\ref{thm:pm_unbiased}.}
If Waiter would want to win the unbiased perfect matching game
on $K_n$ within $\frac{n}{2}$ rounds, %then 
she would need a matching $M$ of size $\frac{n}{2}-1$
after $\frac{n}{2}-1$ rounds. However, since there is only
one possible edge to extend $M$ to a perfect matching,
but Waiter needs to offer two edges,
Client can easily prevent the perfect matching in round $\frac{n}{2}$.
Hence, $\tau_{WC}(\mathcal{PM}_n,1)\geq \frac{n}{2}+1$ follows.
For equality we just apply Theorem~\ref{WCUPM} with $H=\varnothing$ and let Waiter solely
play on a subgraph of $K_n$ isomorphic to $K_{\frac{n}{2},\frac{n}{2}}$.
\hfill $\Box$

% Unbiased Hamiltonicity game
\section{Unbiased Hamiltonicity game}\label{sec:ham}

In the section we will prove Theorem~\ref{thm:ham_unbiased} by showing a slightly stronger statement which will also be applied later in the discussions of the pancyclicity game (Section~\ref{sec:pancyclic})
and the tree embedding game (Section~\ref{sec:trees}).

\begin{thm}\label{WCUHC}
For large enough $n$, in the unbiased Waiter-Client game on $K_n$, Waiter has a strategy to force a Hamilton cycle $H$ within $n+1$ rounds such that the following properties hold
immediately after $H$ is created:
\begin{enumerate}
 \item $\forall v \in V(K_n):d_W(v) < 10$.
 \item Let $e_1^C$ be the first edge Client claims in the game, 
 	then $e_1^C \in E(H)$.
 \item There exists a path $P \subset C$ of length $\frac{1}{5} n$ 
    s.t. $E_W(V(P)) = \emptyset$.
	\end{enumerate}
\end{thm}

\begin{proof}
At the beginning of the game, let Waiter fix an arbitrary
subset $A_1 \subset V$ of size $4$. 
Waiter's strategy will be as follows: At first she forces four vertex disjoint paths
$P_i$ (with $i\in [4]$) in Client's graph, each having an endpoint in $A_1$, 
such that these paths cover the whole vertex set $V=V(K_n)$. 
Afterwards, she makes Client connect the mentioned paths 
to a Hamilton cycle such that the prescribed properties are satisfied.

Let $A_1=\{a_i:~ i\in [4]\}$ and initially, for every $i\in [4]$, let 
$P_i$ be the path consisting only of the vertex $a_i$.
Waiter will force Client to extend the path $P_i$ for every $i\in [4]$, such that $a_i$ remains one of its endpoints, until the union of these four paths covers $V$. 
At any point of the game, we let $\mathcal{P}$ denote the collection of these four paths. We set 
$V(\mathcal{P}) = \bigcup_{i\in [4]} V(P_i)$ and $R=V\setminus V(\mathcal{P})$. Moreover, we always denote
with $a_i'$ the other endpoint of $P_i$ different from $a_i$
(except when $v(P_i)=1$ where we set $a_i'=a_i$),
and we set $A_2 = \{a_i':~ i\in [4]\}$. During most of the game, Waiter's strategy is to consider the paths in pairs. She takes two turns to extend either $P_1$ and $P_2$ or $P_3$ and $P_4$ and does so alternately. In order to keep our notation short, we define $\pi$ to be the permutation on $[4]$ with cycles $(1~2)$ and $(3~4)$, and we sometimes denote $P_4$ with $P_0$ when we consider indices modulo 4.

In the following, we will present a strategy for Waiter and then prove that this strategy allows her to force a Hamilton cycle
within $n+1$ rounds such that all the prescribed properties are ensured. 
If at any point during the game, she is unable to follow her strategy, she forfeits the game. (We will later see that this does not happen.) The sets $A_1, A_2, C, W, R$ and $\mathcal{P}$ are updated at the end of every turn.
Waiter's strategy consists of the following three stages: \\

\textbf{Stage I:} This stage lasts exactly $n-4$ rounds. 
During this stage Waiter extends the four paths $P_i$ until $R=\emptyset$. She does this by alternating between two types of moves:
\begin{enumerate}
\item[] \textbf{Type A:} Let this be the $i^{\text{th}}$ round, and $x\in R$ be a vertex maximizing $d_W(x,V(\mathcal{P}))$ (breaking ties arbitrarily). Then Waiter offers the edges
$xa_i'$ and $xa_{i+1}'$ (with indices taken mod~4). After Client
has chosen one of these edges and thus extended one of the paths $P_i$ or $P_{i+1}$ (indices taken mod~4), the sets $A_2, C, W, R$ and $\mathcal{P}$ are updated in the obvious way.

\vspace{0.1cm}

\item[] \textbf{Type B:} Let this be the $i^{\text{th}}$ round, and let $x,y \in R$ be picked arbitrarily.
Moreover, let $P_t$ be the path which was extended in the previous round. Then Waiter offers the edges $xa_{\pi(t)}'$ and $ya_{\pi(t)}'$. After Client
has chosen one of these edges and thus extended the path $P_{\pi(t)}$, the sets $A_2, C, W, R$ and $\mathcal{P}$ are updated in the obvious way.
\end{enumerate} 

As long as $|R|\geq 2$ holds, Waiter alternates between these two types of moves, with Type~$A$ being considered in odd rounds and Type~$B$ being considered in even rounds. Once $|R|=1$ holds, Waiter plays one more round according to Type~A. Afterwards, she proceeds with Stage II.

\medskip

\textbf{Stage II:} This stage lasts exactly $2$ rounds, in which Waiter forces Client to connect the paths from $\mathcal{P}$. As long as $|\mathcal{P}| > 2$, Waiter connects two paths in Client's graph as follows: She fixes a vertex $v \in A_2$ such that $d_W(v,A_2)$ is maximal and offers $vx, vy$ where $x,y \in A_1$ are picked such that they do not belong to the same path as $v$. W.l.o.g. assume that Client claims $vx$ and thus connects two paths $P_{j_1},P_{j_2}\in\mathcal{P}$. Then update $A_1$ and $A_2$ by removing $v$ and $x$ respectively, and update $\mathcal{P}$ by removing $P_{j_1}$ and $P_{j_2}$, while adding the path induced by $E(P_{j_1})\cup E(P_{j_2})\cup \{vx\}$. 

Once $|\mathcal{P}|=2$ holds, Waiter proceeds with Stage~III. 

\medskip

\textbf{Stage III:} Within $3$ rounds Waiter forces a Hamilton cycle as desired. The details of how she can do this can be found in the strategy discussion.

\medskip

It is evident that, if Waiter can follow the strategy without forfeiting the game, she creates a Hamilton cycle $H$ within $n+1$ rounds.
It thus remains to check that she does not forfeit the game
and that $H$ fulfills the properties (1) -- (3) from Theorem~\ref{WCUHC}.

\medskip

\textbf{Strategy discussion:}

\textbf{Stage I:} At any point of the game we call a vertex $v\in R$ \textit{bad} if $d_W(v,V(\mathcal{P}))\geq 1$ holds. We observe first that there will never be more than one such vertex which at the same time helps Waiter to follow the strategy of Stage~I.

\begin{obs}\label{obs_UH1}
For every $i\leq n-4$, Waiter can follow the $i^{\text{th}}$ move of the proposed strategy. Moreover, the following holds for every $i\leq n-5$:
\begin{itemize}
\item[(a)] if $i$ is odd, then $e_W(A_2)=1$ and no bad vertex exists at the end of round~$i$. Moreover, the unique edge in $E_W(A_2)$ connects endpoints of $P_i$ and $P_{i+1}$ (indices taken mod~4). 
\item[(b)] if $i$ is even, then $e_W(A_2)=0$ and
there is exactly one bad vertex $z$ at the end of round $i$.
Moreover it holds that $d_W(z,V(\mathcal{P}))=d_W(z,V(P_{i-1})\cup V(P_i))=1$ (indices taken mod~4).
\end{itemize}
Moreover, (c) $e_W(A_2)\leq 2$ at the end of round $i=n-4$.
\end{obs}

\begin{proof}
The statement follows by induction on $i$. 
Waiter can obviously follow the strategy for round 1, where she offers two edges according to Type~A. The edge claimed by Client extends $P_1$ or $P_2$. After the update, the other edge belongs to $E_W(A_2)$ and connects the endpoints of $P_1$ and $P_2$, making sure that statement~(a) holds for $i=1$. Let $i>1$ then.

Consider first the case when $i\leq n-5$ is even, and observe that $|R|\geq 2$ before round $i$. 
In round~$i-1$ Waiter played according to Type~A and extended 
a path $P_{t}$ with $t \equiv i-1~\text{or}~i$~(mod~4). By induction, there was no bad vertex at the end of round $i-1$,
but there was exactly one edge $e^W$ in $E_W(A_2)$, and $e^W$ connected endpoints of $P_{i-1}$ and $P_i$ (indices taken mod~4).
Now, in round $i$ Waiter wants to play according to Type~B and needs to offer two edges $xa_{\pi(t)}'$ and $ya_{\pi(t)}'$ with $x,y\in R$. She can do this since $|R|\geq 2$ and $x,y$ cannot be bad.
The edge claimed by Client extends $P_{\pi(t)}$. By this, $e^W$ is removed from $E_W(A_2)$ after the update, leading to $e_W(A_2)=0$.
The other edge goes to Waiter's graph and creates exactly one bad vertex $z\in \{x,y\}$. Since $i$ is even and thus
$\pi(t) \equiv i~\text{or}~i-1$ (mod~4), it follows that
$d_W(z,V(\mathcal{P}))=d_W(z,V(P_{\pi(t)}))=d_W(z,V(P_{i-1})\cup V(P_i))=1$ (indices taken mod~4).

Consider next the case when $i\leq n-5$ is odd. 
By induction we know that $e_W(A_2)=0$ and there was exactly one bad vertex $z$ at the end of round $i-1$, but
$d_W(z,V(\mathcal{P}))=d_W(z,V(P_{i-2})\cup V(P_{i-1}))=1$ (indices taken mod~4). 
%Oberserve that this 
%also gives $d_W(z,V(P_{i})\cup V(P_{i+1}))=0$ (indices taken mod~4). 
Now, in round $i$, Waiter wants to play according to Type~A.
She picks a vertex $x\in R$ maximizing $d_W(x,V(\mathcal{P}))$,
hence setting $x=z$ by the uniqueness of the bad vertex. She needs to offer the edges $xa_i'$ and $xa_{i+1}'$, which she can do since
$d_W(x,V(P_{i})\cup V(P_{i+1}))=0$ (indices taken mod~4).
The edge claimed by Client extends a path in $\mathcal{P}$ by the vertex $x$, so that $x$ is removed from $R$
and there does not exist a bad vertex anymore.
After the update of $\mathcal{P}$ in round $i$, the edge which goes to Waiter's graph connects the endpoints of $P_i$ and $P_{i+1}$ (indices taken mod~4) belonging to $A_2$, such that $e_W(A_2)=1$ as claimed.

Finally, consider the case when $i=n-4$. Then, Waiter can follow the strategy for round $i$ analogously to the case when $i\leq n-5$ is odd. By induction, using (a) or (b), it holds that 
$e_W(A_2 \cup R)= 1$  at the end of round $i-1$. Since, during round $i$, the last vertex of $R$ is moved to $A_2$ and since Waiter receives only one new edge, it is immediately clear that
$e_W(A_2)\leq 2$ afterwards.
\end{proof}

\medskip

{\bf Stage II:} When Waiter enters Stage~II, Client's graph
is the disjoint union of four vertex disjoint paths covering $V$,
with each path being of length roughly $\frac{n}{4}$, since the pairs
$(P_1,P_2)$ and $(P_3,P_4)$ were extended alternately during Stage~I.
Before we show that Waiter can follow Stage~II of the proposed strategy,
let us first observe how Waiter's edges are distributed at the end of Stage~I.

\begin{obs}\label{obs_UH2}
Right at the moment when Waiter enters Stage~II, the following holds:
\begin{enumerate}
\item[(a)] $E_W(A_1)\cup E_W(A_1,A_2) = \varnothing$ and $e_W(A_2)\leq 2$,
\item[(b)] $d_W(v)\leq 4$ for every $v\in V$,
\item[(c)] $E_W(V(P_i))=\varnothing$ for every $i\in [4]$.
\end{enumerate}
\end{obs}

\begin{proof}
For (a) notice that only in the first four rounds of Stage~I
Waiter offers edges incident to $A_1$, and none of these
edges is contained in $A_1$. All the other endpoints of these edges
are part of $V(\mathcal{P})$ at the end of the $5^{\text{th}}$ round,
since by property (a) of Observation~\ref{obs_UH1}, there do not exist
bad vertices at that moment. But now, since all paths get extended further in Stage~I by attaching edges to the vertices in $A_2$ and making appropriate updates, none of the mentioned endpoints belongs to $A_2$ later on. It thus follows that
$E_W(A_1,A_1\cup A_2)=\varnothing$ at the end of Stage~I. The inequality $e_W(A_2)\leq 2$ is already given by property (c) in Observation~\ref{obs_UH1}.

For (b) observe that in Stage I, immediately after a vertex $v$ is added to
some path $P_i\in \mathcal{P}$, it holds that 
$d_W(v,V(\mathcal{P}))\leq 1$
if $v$ was not bad before, or $d_W(v,V(\mathcal{P}))\leq 2$
if $v$ was bad before. 
This degree may increase further
by at most 2,
when the pair of paths $(P_i,P_{\pi(i)})$ is considered
for a further extension by a sequence of turns of Type A and Type B.
But then, according to the strategy, both paths get extended,
which ensures that from now on $v$ is not an endpoint anymore and
Waiter does not offer any further edges at $v$ throughout Stage~I.

For (c), let $e^W$ be any edge claimed by Waiter in Stage~I.
If this edge was offered by Type~A, then after the corresponding round $i$, the edge $e^W$ belongs to $E_W(V(P_i),V(P_{i+1}))$ (indices taken mod~4). Otherwise, if $e^W$ was offered by Type~B, then 
after the corresponding round $i$, $e^W$ connects the unique bad vertex $z$ with the endpoint of one of the paths $P_{i-1}$ or $P_{i}$.
In the next round, playing according to Type~A, Waiter makes sure that
$z$ is added to one of the paths $P_{i+1}$ or $P_{i+2}$,
leading to $e^W\in E_W(V(P_r),V(P_s))$ with $r\neq s$.
\end{proof}

\medskip

Now, having Observation~\ref{obs_UH2} in hand, one can easily see that Waiter can follow Stage~II of her strategy without forfeiting the game. Indeed, by property~(a) from the observation, we know that all edges between $A_1$ and $A_1\cup A_2$ are free. Hence, she can offer edges $vx$ and $vy$ as required by her strategy. Moreover, we have $e_W(A_2)\leq 2$ at the beginning of Stage~II. Since in Stage~II Waiter always picks $v\in A_2$ such that $d_W(v,A_2)$ is maximized and since $v$ is removed from $A_2$ after the update, it %then 
follows that $e_W(A_1\cup A_2)=0$ must hold at the end of Stage~II.

\medskip

{\bf Stage~III:} When Waiter enters Stage~III, $\mathcal{P}$ consists of two paths, say $P_1$ and $P_2$, such that all the four edges
between their endpoints are still free. Moreover, it holds that $d_W(v)\leq 6$ for every $v\in V$, since these degrees were bounded by 4 at the end of Stage~I and since Stage~II took only 2 rounds.

Now, in Stage III, the first step for Waiter is to force a Hamilton path in Client's graph. To do so, she arbitrarily chooses an endpoint $v$ of $P_1$ and offers the edges $vx,vy$ with $x,y$ being the endpoints of $P_2$. Let $P= (v_1, v_2, \dots , v_n)$ be the Hamilton path that is created in Client's graph by this first move. Then by the
conditions from the beginning of Stage~III we know that $v_1v_n$ is still unclaimed. 

Now, by using P\'osa rotations~\cite{Posa1976}, Waiter forces a Hamilton cycle in Client's graph. For her second move in Stage~III Waiter then picks two vertices $v_i$ and $v_j$ with $i,j\notin \{1,n\}$ such that they are not neighbours of each other, and such that $e_1^C \notin \{v_iv_{i+1},v_jv_{j+1},\}$ where $e_1^C$ denotes the edge claimed by Client in round 1, and such that
the edges $v_1v_{i+1},v_1v_{j+1},v_iv_n,v_jv_n$ are still free. 
Such vertices must exist since Waiter's degree of all vertices is bounded by $7$ at this moment. She offers $v_iv_n$ and $v_jv_n$ to Client, who w.l.o.g. claims $v_iv_n$. In the last round Waiter then offers $v_1v_{i+1}$ and $v_1v_n$, and no matter which edge Client chooses, that edge closes a Hamilton cycle $H$. 

Hence, in order to finish our argument, it remains to prove that the properties of Theorem~\ref{WCUHC} hold. 

Property~(1) holds since for every $v\in V$ we had $d_W(v)\leq 6$
at the beginning of Stage~III, while Stage~III lasted exactly 3 rounds. 
Property~(2) holds because Client's only edge which is not part of $H$ is either $v_iv_{i+1}$ or $v_iv_n$, depending on which edge Client claimed in the last round, and $v_i$ was chosen in such a way that both edges differ from $e_1^C$.
For Property~(3) recall that, during Stage I, among the paths $P_1,\ldots,P_4$ there were no interior Waiter's edges, according to Observation~\ref{obs_UH2}, and each of these paths reached a length longer than $\frac{1}{5}n$. Also, when Waiter connects these paths during Stage~II and Stage~III no such interior edges are created, as Waiter only offers edges between endpoints. Moreover, when we remove the unique Client's edge which does not belong to $H$, at most one of these paths from Stage~I can get destroyed, and hence there must remain at least three paths
supporting Property~(3).
\end{proof}

\textit{Proof of Theorem~\ref{thm:ham_unbiased}.}
If Waiter would want to win the unbiased Hamiltonicity game
on $K_n$ within $n$ rounds, %then 
she would need a Hamilton path of length $n-1$
after $n-1$ rounds. However, since there is only
one possible edge to extend this to a Hamilton cycle %,
and Waiter needs to offer two edges,
Client can easily prevent the Hamilton cycle in round $n$.
Hence, $\tau_{WC}(\mathcal{H}_n,1)\geq n+1$ follows.
For equality we just apply Theorem~\ref{WCUHC}.
\hfill $\Box$

% Unbiased Pancyclicity game 
\section{Unbiased Pancyclicity game}\label{sec:pancyclic}

\begin{proof}[Proof of Theorem~\ref{thm:pancyclic_fast}]
Set $g(n)=\lceil \log_2^{(k)}(n) \rceil$ and
$f(n)=g(n) + 100$ for any positive integer $k$. 
In the following we will describe a strategy for Waiter in the unbiased Waiter-Client game on $K_n$, and afterwards we will show that it is a strategy with which Waiter forces a pancyclic spanning subgraph of $K_n$ within at most 
$n+ \log_2(n) + O(\max\{f(n),k\})$ rounds. Whenever Waiter is not able to follow the proposed strategy, she forfeits the game. (We will show later that this does not happen.)
The strategy is split into the following five stages:

\medskip

{\bf Stage I:} Within at most $n+1$ rounds,
Waiter forces a Hamilton cycle $H = (v_1,v_2,\ldots,v_n)$ such that the following holds right at the moment when the Hamilton cycle is completed:
\medskip

\begin{minipage}{0.95\textwidth}
\begin{enumerate}[label=\itmarab{H}]
\item\label{ham:degrees} $d_W(v)<10$ for every $v\in V(K_n)$,
\item\label{ham:subpath} for every $1\leq j_1,j_2 \leq \frac{n}{10}$ with $|j_1-j_2|\geq 2$ it holds that $v_{j_1}v_{j_2}\notin W\cup C$.
\end{enumerate}
\end{minipage}

\medskip

The details can be found in the strategy discussion. Afterwards, Waiter proceeds with Stage~II.

\medskip

{\bf Stage II:} This stage lasts two rounds. At first, Waiter offers the edges
$v_iv_{f(n)+i}$ with $i\in [2]$.
Among these edges, Client needs to claim one;
denote it with $w_1w_{f(n)+1}$, 
and afterwards let
$$
w_i=
\begin{cases}
v_i & ~ \text{if }w_1=v_1\\
v_{i+1} & ~ \text{if }w_1=v_2\\
\end{cases}
$$
for every $i\in [n]$ (with $v_{n+1}:=v_1$).
In the second round, Waiter offers two free edges between $w_{f(n)+1}$ and $\{w_{n-60},\ldots,w_{n-50}\}$, among which Client needs to choose one. Afterwards, Waiter proceeds with Stage~III.

\medskip

{\bf Stage III:} This stage lasts $f(n)-2$ rounds. 
In the $i^{\text{th}}$ round of Stage~III,
Waiter offers the edges $w_1w_{i+2}$ and
$w_{f(n)-i}w_{f(n)+1}$, among which Client always 
needs to claim one.
Once all the $f(n)-2$ rounds of Stage~III are played, Waiter proceeds with Stage~IV.

\medskip

{\bf Stage~IV:} This stage lasts at most 
$\lceil \log_2(n) \rceil$ rounds. Waiter makes sure that at the end of the $i^{\text{th}}$ round of Stage~IV, there exist vertices
$w_{t_0},w_{t_1},\ldots,w_{t_i}$ 
such that the following holds:
\medskip

\begin{minipage}{0.95\textwidth}
\begin{enumerate}[label=\itmarab{W}]
\item\label{W:order} $f(n)+1=t_0<t_1<\ldots<t_i\leq n$,
\item\label{W:edges} $w_{t_{j-1}}w_{t_{j}}\in C$ for every $j\in [i]$,
\item\label{W:distance} 
$\min\{2t_{i-1}  - 2i,n\} - 20 \leq t_i \leq  
\min\{2t_{i-1} - 2i,n\}$.
\end{enumerate}
\end{minipage}

\medskip

In order to do so, in the $i^{\text{th}}$ round
Waiter offers two free edges of the form
$w_{t_{i-1}}w_j$ with 
$\min\{2t_{i-1}  - 2i,n\} -20 \leq j \leq  
\min\{2t_{i-1} - 2i,n\}$. For the edge
$w_{t_{i-1}}w_j$ chosen by Client,
Waiter then sets $t_i:=j$.
Once there is a round $s$ such that
$n-20 \leq t_s\leq n$ holds, Waiter stops with Stage~IV and proceeds with Stage~V.

\medskip

{\bf Stage~V:} This stage lasts at most $k-1$ rounds. For her $i^{\text{th}}$ move, 
Waiter aims to make Client claim an edge
$w_1w_{\ell}$ with $2 \log_2^{(i)}(n) \leq t_j \leq \ell \leq  t_{j} + 20 \leq  10 \log_2^{(i)}(n)$ for some $j\leq s$. In case Client does not already 
possess such an edge, Waiter just offers two free edges of the mentioned kind. Otherwise, Waiter just skips that move and proceeds to her next move.

\medskip

In the following discussion, we need to check two properties for the given strategy: (1) Waiter can always follow the proposed strategy without forfeiting the game, and (2) when Stage~V is over, Client's graph is pancyclic. Just note that then a pancyclic graph will be forced within at most
$$
n + \log_2(n) + f(n) + k 
= n + (1+o(1))\log_2(n)
$$
rounds, according to the bounds on the number of rounds given in the descriptions of the stages.\\

{\bf Strategy discussion:}

{\bf (1) -- Following the strategy:} Waiter can follow Stage~I because of Theorem~\ref{WCUHC}. 
According to that theorem, Waiter can force a Hamilton cycle $H$ within $n+1$ rounds such that Property~\ref{ham:degrees} holds immediately after $H$ is created. Moreover, she can make sure that right at this moment there is a path $P\subset H$ of length $\frac{n}{5}$ such that
$E_W(V(P))=\varnothing$ holds. Split $P$ into two subpaths $Q_1$ and $Q_2$ of length $\frac{n}{10}$ each.
Since $e(C\setminus H)=1$ holds at the end of round $n+1$, we know that there must be some $i\in [2]$ with $E_C(V(Q_i))\setminus E(Q_i)=\varnothing$. Labelling the vertices of $H$ in such a way that $V(Q_i)=\{v_1,\ldots,v_{\frac{n}{10}}\}$ holds, we obtain Property~\ref{ham:subpath}.

Afterwards, in Stage~II and in Stage~III, Waiter
needs to offer several edges contained in 
$E(\{v_i: i\leq f(n)+2\})\setminus E(H)$,
which are still free by Property~\ref{ham:subpath}
and since $f(n)+2<\frac{n}{5}$.
She also needs to offer two edges between
$w_{f(n)+1}$ and $\{w_{n-60},\ldots,w_{n-50}\}$
which is possible since $d_W(w_{f(n)+1})<10$ at the end of Stage~I.

Next consider Stage~IV and observe the following: 
if Waiter can follow this part of her strategy and as long as $t_i<n-20$ holds, we have
$t_i\geq 2t_{i-1} - 2i - 20$ and $t_0\geq f(n) > 100$,
leading to
\begin{equation}\label{eq:rounds}
t_i>2^i+i^2+50
\end{equation}
by a simple induction.
Thus, Stage~IV lasts at most $\lceil \log_2(n) \rceil$ rounds.
For showing that Waiter can follow each of these rounds, we proceed by induction on $i$: Let us consider the $i^{\text{th}}$ round of Stage~IV  (when $w_{t_0},\ldots,w_{t_{i-1}}$ are already given,
and \ref{W:order}--\ref{W:distance} hold for $i-1$). 
At the end of Stage~I we had $d_W(w_{t_{i-1}})<10$ and, since Stage~II lasted two rounds,
we have $d_W(w_{t_{i-1}})<12$ at the end of Stage~II. Since afterwards (in Stage~III-IV) until the current round, each offered edge was 
incident to some $w_{\ell}$, $\ell<t_{i-1}<\min\{2t_{i-1}-2i,n\}-20$,
there need to be at least two free edges as required by the strategy description.
Once Client has claimed one of these edges, it is obvious that the Properties~\ref{W:order}--\ref{W:distance} hold again for $i$.

Finally, consider the $i^{\text{th}}$ round of Stage~V
for $i\in [k-1]$. Since $t_0=f(n)+1$, $t_s\geq n-20$ and since $t_{j+1}< 2t_j$
for all $j\leq s$, it follows that there must be some $j\in [s]$
with $2\log_2^{(i)} (n) \leq t_j \leq 5 \log_2^{(i)} (n)$. Having such $t_j$ fixed, it is enough to find two free edges $w_1w_{\ell}$ with 
$t_j\leq \ell \leq t_j+20$. This is possible, because $d_W(w_1)<10$ at the end of Stage~I and since in Stage~II-IV no such edge was offered.

\medskip

{\bf (2) -- Finding pancyclicity:}
Let $H=(w_1,w_2,\ldots,w_n)$ be the Hamilton cycle from Stage~I. It is the edge disjoint union of
two paths $P_1=(w_1,w_2\ldots,w_{f(n)+1})$ and
$P_2=(w_{f(n)+1},\ldots,w_{n-1},w_n,w_1)$ of lengths
$f(n)$ and $n-f(n)$, respectively.
Both paths are closed to cycles in Client's graph by the edge $w_1w_{f(n)+1}$ which was claimed in Stage~II. We next observe that
%, 
after Stage~III the following holds:

\begin{obs}\label{short.paths}
For every $0\leq t\leq f(n)-1$ there is a path $P_1^t\subset C$ such that
\begin{enumerate}
\item[(i)] $V(P_1^t)\subset V(P_1)$,
\item[(ii)] $P_1^t$ has length $f(n)-t$,
\item[(iii)]
$w_1$ and $w_{f(n)+1}$ are the endpoints of $P_1^t$~ .
\end{enumerate}

\end{obs}

\medskip

\begin{proof}
For $t=0$ and $t=f(n)-1$ we let
$P_1^{f(n)-1}$ consist %s 
of the edge $w_1w_{f(n)+1}$ and $P_1^{0}=P_1$. For every $1\leq t\leq f(n)-2$ Client claimed either $w_1w_{t+2}$ or 
$w_{f(n)-t}w_{f(n)+1}$ in round $t$ of Stage~III,
and thus we can choose either
$P_1^t=(w_1,w_{t+2},w_{t+3},\ldots,w_{f(n)},w_{f(n)+1})$ or
$P_1^t=(w_1,w_2,\ldots,w_{f(n)-t-1},w_{f(n)-t},w_{f(n)+1})$.
\end{proof}

\medskip\medskip

Let $w_{f(n)+1}w_{n-p}$ be the edge claimed by Client in the second round of Stage~II, and observe that $50\leq p < f(n)$.
By closing the above paths $P_1^t$ into cycles,
either using the edge $w_1w_{f(n)+1}$ or
the path $(w_{f(n)+1},w_{n-p},w_{n-p+1},\ldots,w_n,w_1)$,
we obtain cycles of all lengths between 
$3\leq \ell \leq f(n)+p$.
Hence, it remains to find cycles of all the 
lengths larger than $f(n)+p \geq f(n) + 50$. 
In order to do so, we will fix $0\leq m \leq k-1$
from now on and we will explain 
how we find cycles of all lengths between
$\log_2^{(m+1)}(n)+50$ and $\min\{2\log_2^{(m)}(n),n\}$ in Client's graph. Running over all possible $m$ finishes the argument, 
% then, 
as the interval $[3,\log_2^{(k)}(n)+50]$ and 
all the intervals
$[\log_2^{(m+1)}(n) + 50,\min\{2\log_2^{(m)}(n),n\}]$ with $0\leq m\leq k-1$ cover all integers from $3$ to $n$.

\medskip

Having $m$ fixed, set $w_{k_m}=w_n$ if $m=0$ and otherwise 
let $w_{k_m}$ be the vertex $w_{\ell}$ from the $m^{\text{th}}$ move in Stage~V.
By Stage~IV (in case $m=0$) or Stage~V (in case $m\neq 0)$ there is some index $j_m\leq s$
such that $k_m-20\leq t_{j_m}\leq k_m$.
Moreover, set $a_i:=t_i-t_{i-1}-1$ 
to be the number of vertices between 
$w_{t_{i-1}}$ and $w_{t_i}$ on $P_2$, for every $i\leq s$. Then, at the end of Stage~V the following holds:

\begin{obs}\label{many.cycles}
For every subset $S\subset [j_m]$ there is a path $P_2^S\subset C$
such that
\begin{enumerate}
\item[(i)] $V(P_2^S)\subset V(P_2)$,
\item[(ii)] $P_2^S$ has length 
$k_m-f(n)-\sum_{i\in S} a_i$,
\item[(iii)]
$w_1$ and $w_{f(n)+1}$ are the endpoints of $P_2^S$~ .
\end{enumerate}

\end{obs}

\medskip

\begin{proof}
If we extend the subpath
$(w_{f(n)+1},\ldots,w_{k_m})$ from $P_2$ by Client's edge $w_{k_m}w_1$, we obtain a path $P_2'$
of length $k_m-f(n)$. By replacing any subpath $(w_{t_{i-1}},\ldots,w_{t_i})$, where $i\leq j_m$,
with the edge $w_{t_{i-1}}w_{t_i}$ which was claimed in Stage~IV,
the path $P_s'$ can be shortened by exactly $a_i$ edges. As this can be done for any $i\in S$, 
we can shorten $P_2'$ to a path 
of length $k_m-f(n)-\sum_{i\in S}a_i$. This proves the observation.
\end{proof}

\medskip\medskip

Now, by joining the path $P_1^t$ with the path $P_2^S$, for any $0\leq t\leq f(n)-1$ 
and any $S\subset [j_m]$,
we obtain a cycle of length
$k_m-(t+\sum_{i\in S} a_i)$. 
We will see in the following that this will indeed give us cycles
of all lengths between
$\log_2^{(m+1)}(n)+50$ and $\min\{2\log_2^{(m)}(n),n\}$.
We start with the following observation.

%\begin{center}
%\includegraphics[scale=0.8,page=2]{pancyclic_pic}
%\end{center}

\begin{obs}\label{obs:integers}
Every integer in $[f(n)-1+\sum_{i\in [j_m]} a_i]$
can be written in the form 
$t+\sum_{i\in S} a_i$ with 
$0\leq t\leq f(n)-1$  and $S\subset [j_m]$.
\end{obs}

\begin{proof}
Inductively, one may show that for every 
$0\leq j\leq j_m$
\begin{enumerate}
\item[(S)] every integer in 
$[f(n)-1+\sum_{i\in [j]} a_i]$ can be written 
as a sum $t+\sum_{i\in S} a_i$ with 
$0\leq t\leq f(n)-1$  and $S\subset [j]$~ .
\end{enumerate}
The beginning of the induction ($j=0$) should be obvious.
So, let $j+1>0$, and assume (S) to be true for $j$.
By~\ref{W:distance}, the definition of $a_j$ and since $t_0=f(n)+1$, it follows that
$$a_{j+1} = t_{j+1} - t_j - 1
	\leq t_j - 2(j+1) \leq \sum_{i\in [j]} a_{i} + f(n) - 1  ~ .$$
As, by induction, the integers up to the last sum can be written
as  $t+\sum_{i\in S} a_i$ with 
$0\leq t\leq f(n)-1$ and $S\subset [j]$,
adding $a_{j+1}$ to the latter creates all integers
in
$$
\left[a_{j+1},\sum_{i \in [j+1]} a_i + f(n)-1\right]
\supset 
\left[\sum_{i\in [j]} a_{i} + f(n) - 1 , 
\sum_{i\in [j+1]} a_i + f(n)-1\right]~ .
$$
Note that the last set contains all the remaining integers for completing the induction step.
This shows (S) for $j+1$ and finishes the proof of the observation.
\end{proof}

\medskip\medskip

Finally, observe that
$
f(n) - 1 +\sum_{i\in [j_m]} a_i
=
t_{j_m} - j_m - 2
$
and hence, by Observation~\ref{obs:integers} and by the argument immediately after Observation~\ref{many.cycles},
we see that we can find cycles of all lengths
between $k_m- (t_{j_m} - j_m - 2)$ and $k_m$.
Now, note that $k_m\geq t_{j_m}\geq k_m-20$ by the choice of $k_m$, and $k_m\geq \min\{2\log_2^{(m)}(n),n\}$
since $t_{j_m}\geq 2\log_2^{(m)}(n)$ by Stage~V (in case when $m\neq 0$).
Moreover, using that 
$t_{j_m}\leq 10\log_2^{(m)}(n)$
and (\ref{eq:rounds}) hold, we get 
$j_m\leq \log_2^{(m+1)}(n) + 10$.
Hence, we obtain all cycle lengths between
$\log_2^{(m+1)}(n) + 50 \geq  k_m- (t_{j_m} - j_m - 2)$ and $\min\{2\log_2^{(m)}(n),n\}\leq k_m$,
as desired.
\end{proof}

{\bf Remark:} In the above argument, we need that the interval $[3,\log_2^{(k)}(n)+50]$ and 
all the intervals
$[\log_2^{(m+1)}(n) + 50,\min\{2\log_2^{(m)}(n),n\}]$ with $0\leq m\leq k-1$ cover all integers from $3$ to $n$. Hence we only need to ensure that
$2\log_2^{(m)}(n)\geq \log_2^{(m)}(n) + 50$
holds for all $m\leq k-1$, i.e. $\log_2^{(m)}(n)\geq 50$, which is given when 
$\log_2^{(k+2)}(n)\geq 2$.
Thus, if we choose $H(n)$ to be the smallest integer $t$ such that $\log_2^{(t)}(n)<2$ holds, then 
the above proof gives us that the game is won within $n+\log_2(n) + H(n) + O(1)$ rounds.
This is only an additive constant away from the best known general upper bound on the minimal size of pancyclic graphs as mentioned in~\cite{Bondy1971}.

% Unbiased Tree game & Tree factor game
\section{Unbiased games involving trees}\label{sec:trees}

In this section we will prove Theorem~\ref{thm:tree}.
Based on ideas from~\cite{CFGHL2015} and~\cite{FHK2012}, we will split 
the given tree $T$ into a subgraph $T'$ and a nice behaving structure 
(large matching or long path). In her strategy, Waiter 
%then 
will first
force a copy of $T'$ more or less greedily and without wasting any move,
while additionally caring about the distribution of her edges.
Afterwards, Waiter %then 
will force the appropriate nice behaving structure
while wasting at most one round. 

\medskip

Let $T$ be any tree. We denote by $L=L(T)$ the set of leaves of $T$ and by 
$N_T(L)$ the set of vertices which are in the neighbourhood of the leaves w.r.t. $T$.
For every $x \in N_T(L)$ let $\ell(x)$ be the number of leaves which are neighbours of $x$ in $T$ and define $\Delta(N_T(L)) = \max_{x \in N_T(L)} \ell(x)$.

\medskip

We start with the following lemma, similar to Lemma~2.1 in \cite{Kriv2010},
which states that each of the trees $T$ considered in Theorem~\ref{thm:tree}
has a nice behaving structure as mentioned above:
a large matching where every edge is incident to a leaf
or a long \textit{bare path}, i.e. a path such that all the inner vertices have degree $2$ in $T$.

\begin{lemma}\label{lem:tree-property}
For every $\mu\in (0,1/2)$ there exists $\epsilon>0$ such that the following holds for every large enough integer $n$:
Let $T$ be a tree on $n$ vertices with 
%\textcolor{red}{define!!!} 
$\Delta (N_T(L)) \leq \epsilon \sqrt{n}$, then either $|N_T(L)|\geq \mu \sqrt{n}$ or $T$ contains a bare path of length at least $\mu \sqrt{n}$.
\end{lemma}

\begin{proof}
Set $\epsilon=\mu /3$ and assume that $|N_T(L)|<\mu \sqrt{n}$.
We will show now that $T$ needs to contain a bare path of length at least $\mu\sqrt{n}$.
By our assumption, we obtain
\begin{align*}
|L| \leq \Delta(N_T(L))\cdot |N_T(L)| < \epsilon \mu n~ .
\end{align*}
Now, let $T'=T - L$ then $n':=|V(T')| \geq n-\mu \epsilon n > \frac{2n}{3}$.

    Let $S_i= \{v \in V(T')$ $|$ $d_{T'}(v)=i\}$ and $S_{\geq i}= \{v \in V(T')$ $|$ $d_{T'}(v) \geq i\}$ for every $i\in [n]$, and observe that $S_1 \subseteq N_T(L)$. Further, let $\mathcal{P}$ be the collection of maximal bare paths in $T'$ and let $\tilde{T}$
    be the tree obtained from $T'$ by contracting each path in $\mathcal{P}$ to an edge. Then
    \begin{equation*}
        |\mathcal{P}| = e(\tilde{T})
        	= v(\tilde{T}) - 1 = |S_1| + |S_{\geq 3}| - 1.
    \end{equation*}
    Also, by the Handshake Lemma it holds that
    \begin{align*}
        2(n'-1) = 2e(T') &\geq |S_1|+ 2|S_2|+3|S_{\geq 3}| \\
        &  = 2 (|S_1| + |S_2|+ |S_{\geq 3}|)+ |S_{\geq 3}| - |S_1| 
           = 2n' + |S_{\geq 3}| - |S_1| ,
    \end{align*}
    leading to $|S_{\geq 3}| < |S_1|$ and hence
    \begin{equation*}
    |\mathcal{P}| < |S_1| + |S_{\geq 3}| 
    			< 2|S_1| \leq 2|N_T(L)| < 2\mu\sqrt{n} 
    		\leq \mu \sqrt{6n'}~ .
    \end{equation*}
	By the Pigeonhole Principle and since each vertex from $S_2$
	belongs to exactly one path in $\mathcal{P}$, 
	there exists a bare path of length at least
    \begin{equation*}
    \frac{|S_2|}{|\mathcal{P}|}= \frac{n'-|S_1|-|S_{\geq 3}|}   
    {\mathcal{|P|}} \geq 
    \frac{n'-\mu \sqrt{6n'}}{\mu \sqrt{6n'}}
    > \mu \sqrt{n}
    \end{equation*}
    where last inequality uses that $\mu < \frac{1}{2}$,
    $n'>\frac{2}{3}n$ and that $n$ is 
    large enough.
\end{proof}

Theorem~\ref{thm:tree} will follow from the next slightly stronger result
which will be used later as well for the study of the tree-factor game.

\begin{thm}\label{thm:tree_stronger}
There exists $\epsilon>0$ such that the following holds for every large enough integer $n$:
Let $T$ be a tree on $n$ vertices and let $v\in V(T)\setminus (L\cup N_T(L))$ be such that the following holds:
\begin{enumerate}
    \item $d_T(v) \leq \frac{n}{3}$ and ,
    \item $\Delta (T \backslash \{v\} ) \leq \epsilon \sqrt{n}$.
\end{enumerate}
Moreover, let $p\in V(K_n)$. Then, in an unbiased Waiter-Client game on $K_n$, Waiter has a strategy to force Client to claim a copy of $T$ within $n$ rounds such that
\begin{itemize}
\item[(a)] in Client's copy of $T$, the vertex $p\in V(K_n)$ represents the vertex $v\in V(T)$, and
\item[(b)] in each round of her strategy, Waiter offers either 2 edges or no edge incident to $p$.
\end{itemize}
\end{thm}

\begin{proof}
Let $\mu = \frac{1}{3}$ and choose $\varepsilon < \frac{\mu}{20}$ 
according to Lemma~\ref{lem:tree-property}. 
Then, given a tree $T$ with the properties from the theorem above,
there exists a bare path of length at least $\mu\sqrt{n}$ in $T$ (Case A) 
or we have $|N_T(L)| \geq \mu \sqrt{n}$ (Case B). 
We provide a different strategy for Waiter for each case.

\medskip

In order to describe Waiter's strategy,
we use notation similar to that from~\cite{CFGHL2015} and~\cite{FHK2012}. 
 Let $S \subseteq V(T)$ be an arbitrary set, then an $S$-\textit{partial embedding} of $T$ in $G$ is an injective mapping $f: S \rightarrow V(G)$ such that $f(x)f(y)$ is an edge in  $G$ whenever $xy$ is an edge in $T$. Vertices in $S$ are called \textit{embedded} vertices. 
Let any subgraph $T'\subset T$ be given. 
Then a vertex $v \in f(S)$ is called \textit{closed} with respect to $T'$ if all the neighbours of $f^{-1}(v)$ in $T'$ are embedded as well. Otherwise $v$ is called \textit{open} w.r.t. $T'$. With $\mathcal{O}_{T'}$ we denote the set of all vertices that are open w.r.t. $T'$. Moreover, the vertices of the set $A:=V(G) \setminus f(S)$ are called \textit{available}.

\medskip

If at any point during the game, Waiter is unable to follow the strategy, she forfeits the game. (We will see later that this does not happen.)

\medskip\medskip\medskip\medskip

\textbf{Case A -- Long bare path.} Consider first the case when $T$ contains a bare path of length at least $\mu \sqrt{n}$. 
Let $P$ be such a path of length $\mu \sqrt{n}$, and denote its endpoints 
with $u$ and $w$. Let $u_1$ and $w_1$ be the neighbours of $u$ and $w$ in $P$ respectively. $T\backslash P$ is a forest with two tree components, say $T_1$ and $T_2$. We let $T'\subset T$ be the forest induced by $E(T_1)\cup E(T_2)\cup \{uu_1,ww_1\}$. W.l.o.g. we may assume that both $v$ and $u$ belong to $T_1$.

\medskip

In broad terms, Waiter's strategy is to first force a copy of $T'$ 
(Stage I and II) and then to force a copy of the bare path $P$ (Stage III) 
in such a way that a copy of $T$ is created within $n$ rounds. Throughout the game, she maintains a set $S$ and an $S$-partial embedding $f$ of $T$ into $K_n$
in order to represent the subgraph of $T$ which currently is isomorphic to Client's graph. 
Initially, set $S = \{v, w\}$, $f(v)=p$ and $f(w)=q$ for arbitrary $p,q \in V(K_n)$.
Waiter's strategy is split into the following stages:

\medskip

\textbf{Stage I:} This stage lasts for $d_T(v)$ rounds in which Waiter closes the vertex $v$ w.r.t.~$T$. Each round is played as follows: 

First, Waiter fixes an arbitrary vertex $t\in N_T(v)\setminus S$.  
Waiter then offers two edges $pa_1,pa_2$ such that both edges are free and $a_1,a_2 \in A$. By symmetry, assume that Client chooses the edge $pa_1$. 
Then Waiter updates $A$, $S$ and $f$
by removing $a_1$ from $A$, adding $t$ to $S$ and setting $f(t):=a_1$.

\medskip

\textbf{Stage II:} This stage lasts $e(T')-d_T(v)$ rounds in which
it is Waiter's goal to create a $V(T')$-partial embedding. For each round, she
plays as follows: 

If $S= V(T')$ holds, Waiter proceeds to Stage III. 
Otherwise, fix an arbitrary vertex $x\in f(S\setminus \{u_1,w_1\}) \cap \mathcal{O}_{T}$ and let $t=f^{-1}(x)$. Since $x$ is open, there exists a vertex 
$z \in (V(T_1) \cup V(T_2) \cup \{u_1,w_1 \}) \setminus S$ such that $tz \in E(T')$.  Waiter then offers two free edges $a_1x$ and $a_2x$ such that
$a_1,a_2 \in A$. By symmetry, we may assume that Client chooses $a_1x$. 
Then Waiter updates $A$, $S$ and $f$ by removing $a_1$ from $A$, adding $z$ 
to $S$ and setting $f(z):=a_1$.  Afterwards, she repeats Stage II.

\medskip
     
\textbf{Stage III:} When Waiter enters Stage III, Client's graph is a copy of
the subgraph $T'$.
Within $n-e(T')$ rounds, Waiter now forces a Hamilton path
on $V(K_n)\setminus f(V(T_1)\cup V(T_2))$ with endpoints $f(u_1)$ and $f(w_1)$.
The details of how Waiter can do this are given later in the strategy discussion.\\

{\bf Strategy Discussion:}
It is obvious that if Waiter can follow the given strategy without forfeiting the game, %then 
she forces a copy of $T$ within at most $d_T(v)+(e(T')-d_T(v))+(n-e(T'))=n$ rounds. Hence, it remains to show that Waiter can indeed do so.
However, before we study each stage separately, let us observe the following:

\begin{obs}\label{obs:tree_caseA}
Throughout Stages I and II, as long as Waiter can follow the proposed strategy, it holds that
\begin{enumerate}
\item[(i)] $|A|\geq \mu \sqrt{n} - 2$ and $e_{C\cup W}(A)=0$,
\item[(ii)] $d_W(x,A)\leq d_C(x)$ for every $x\in f(S)$.
\end{enumerate}
\end{obs}

\begin{proof}
Property (i) is immediately clear. The inequality $|A|\geq \mu \sqrt{n} - 2$ holds, since the strategy for the mentioned stages is to force a copy of $T'$ without wasting any move and since 
$e(T')=e(T_1\cup T_2)+ |\{uu_1, ww_1 \}| = n- \mu \sqrt{n} +2$.
The equation $e_{C\cup W}(A)=0$ holds, since Waiter always only offers edges
that intersect $f(S)$.
For Property~(ii), observe that $d_W(x,A)$ may only increase, when
$x\in f(S)$ (since $e_{C\cup W}(A)=0$) and Waiter offers an edge
between $x$ and $A$. However, when this happens in any of the Stages I and II,
Waiter actually offers two edges between $x$ and $A$, 
which increases $d_C(x)$ by $1$ at the same time. 
Hence, $d_W(x,A)$ cannot become larger than $d_C(x)$.
\end{proof}

\medskip

In the following we check that Waiter always can follow the strategy without forfeiting the game.\\

{\bf Stage I:} 
According to the strategy, Waiter needs to offer $2d_T(v) \leq \frac{2n}{3}$ edges at $v$. She can easily do so, since there exists $n-1$ edges to choose from.

\medskip

{\bf Stage II:} The vertex $z$, which is described in the strategy, exists because of our assumption that $x$ is an open vertex. Moreover, by Observation~\ref{obs:tree_caseA},
we have $d_W(x,A) \leq d_C(x) \leq d_T(f^{-1}(x))\leq \varepsilon \sqrt{n}$, which in turn means that at least 
$|A| - \varepsilon \sqrt{n} > \mu\sqrt{n} - 2 - \varepsilon \sqrt{n} > \frac{\mu}{2} \sqrt{n}$ 
edges between $x$ and $A$ are free. 
Hence there exist %s 
two free edges $xa_1$ and $xa_2$ as desired and Waiter can follow the proposed strategy.

\medskip

{\bf Stage III:} When Waiter enters Stage~III, she has successfully 
managed to force a copy of $T'$. 
Let $A'=V(K_n)\setminus f(V(T_1)\cup V(T_2))$ and observe that
$e_{C\cup W}(A')=0$ holds right at this moment.
Indeed, according to Observation~\ref{obs:tree_caseA} we have $e_{C\cup W}(A)=0$. Moreover, $e_{C\cup W}(\{f(u_1), f(w_1) \},A)=0$ holds, since in Stage~II Waiter always chooses $x$ different from $f(u_1)$ and $f(w_1)$, which in turn ensures that, once these vertices are embedded,
Waiter never offers edges incident to those again. 

\medskip

For Stage~III, Waiter now plays as follows. At first she considers a fake round
which is not played at all but where Waiter pretends that Client claimed the edge $e_C:=f(u_1)f(w_1)$.
Afterwards, she continues according to the strategy from Theorem~\ref{WCUHC}
(with $K_n$ replaced by $K_n[A'] \cong K_{|A'|}$), which ensures that within $|A'|+1$ rounds there is a Hamilton cycle in Client's graph on $A'$ 
which contains the edge $e_C$. Since the first round was a fake round,
this actually means that, within $|A'|=n-e(T')$ rounds, 
Waiter obtains a Hamilton path in $A'$ as desired.

\medskip\medskip\medskip\medskip

\textbf{Case B -- Many leaf neighbours.} Consider next the case when
$|N_T(L)| \geq \mu \sqrt{n}$ holds. Then, there exists a matching $M_0$ of size 
at least $\mu \sqrt{n}$ which consists of edges that are incident to leaves of $T$. Define the sets $L_0:=L\cap V(M_0)$,
$D_i=\{w\in V(T): \dist_{T}(v,w)=i\}$, 
$D_{\text{odd}}:=\bigcup_{i\text{ odd}} D_i$ and 
$D_{\text{even}}:=\bigcup_{i\neq 0 \text{ even}} D_i$.
By the Pigeonhole Principle we have that there is a set
$D_{\text{good}}\in \{D_{\text{odd}},D_{\text{even}}\}$ such that 
$D_{\text{good}}\cap N_T(L_0)$ has size at least $\mu' \sqrt{n}$ with
$\mu':=\mu/3$. Let $M'=\{e\in M_0: e\cap D_{\text{good}}\neq \varnothing\}$, $L'=L\cap V(M')$ and $T'=T-L'$. By the choice of $D_{\text{good}}$ we have that
$|M'|\geq \mu' \sqrt{n}$, $v\notin V(M')$ and 
$\dist_{T'}(x,y)\geq 2$ for every $x,y\in N_T(L')$.

\medskip

In broad terms, Waiter's strategy now is to first force a copy of $T'$ (Stage I and II) and then to force a copy of the matching $M'$ (Stage III) in such a way that a copy of $T$ is created within $n$ rounds. Throughout the game, she again maintains a set $S$ and an $S$-partial embedding $f$ of $T$ into $K_n$
in order to represent the subgraph of $T$ which currently is isomorphic to Client's graph. 
Initially, set $S = \{v\}$ and $f(v)=p$ for an arbitrary $p \in V(K_n)$.
Additionally, at any moment in the game we define $S':=S\cap N_T(L')$.

\medskip

\textbf{Stage I:} This stage lasts for $d_{T'}(v)=d_T(v)$ rounds in which Waiter closes the vertex $v$ w.r.t.~$T'$. Each round is played as follows: 

First, Waiter fixes an arbitrary vertex $t\in N_{T'}(v)\setminus S$.  
Waiter then offers two edges $pa_1,pa_2$ such that both edges are free and $a_1,a_2 \in A$. By symmetry, assume that Client chooses the edge $pa_1$. 
Then Waiter updates $A$, $S$ and $f$
by removing $a_1$ from $A$, adding $t$ to $S$ and setting $f(t):=a_1$.

\medskip

\textbf{Stage II:} 
This stage lasts $e(T')-d_{T'}(v)$ rounds in which
it is Waiter's goal to create a $V(T')$-partial embedding,
while also taking care of the distribution of Waiter's edges
between the open and the available vertices. For each round, she
plays as follows:  

If $S = V(T')$ then Waiter proceeds to Stage III. Otherwise,
Waiter considers the following case distinction:\\

\begin{enumerate}
\item[] {\bf Case 1.} Let there be two vertices $u_1,u_2\in \mathcal{O}_{T'}$, and let $t_1=f^{-1}(u_1)$ and $t_2=f^{-1}(u_2)$. By assumption, there exist vertices $z_1, z_2 \in  V(T') \setminus S$ such that $t_1z_1, t_2z_2 \in E(T')$.  Waiter then picks any vertex $a \in A$ such that $au_1$ and $au_2$ are free, where she prefers vertices satisfying $d_W(a,f(S'))\geq 1$,
 and offers $au_1$ and $au_2$ to Client. 
By symmetry we may assume that Client chooses $au_1$. 
Then Waiter updates $A$, $S$ and $f$ by removing $a$ from $A$, adding $z_1$ to $S$ and setting $f(z_1):=a$ . 

\medskip

\item[] {\bf Case 2.} Let there be only one vertex $u\in \mathcal{O}_{T'}$, but assume that $u\notin f(S')$. Let $t=f^{-1}(u)$.
By assumption, there exists a vertex $z \in  V(T') \setminus S$ such that 
$tz \in E(T')$. Waiter then picks any vertices $a_1,a_2 \in A$ 
such that $a_1u$ and $a_2u$ are free, where she prefers vertices satisfying $d_W(a_i,f(S'))\geq 1$, and offers these two edges to Client. 
By symmetry we may assume that Client chooses $a_1u$. Then Waiter updates $A$, $S$ and $f$ by removing $a_1$ from $A$, adding $z$ to $S$ and setting $f(z):=a_1$.

\medskip

\item[] {\bf Case 3.} Let there be only one vertex $u\in 
\mathcal{O}_{T'}$, and moreover assume that $u\in f(S')$. Let $t=f^{-1}(u)$.
By assumption, there exists a vertex $z \in  V(T') \setminus S$ such that 
$tz \in E(T')$, and among such vertices we choose $z$ such that $d_{T'}(z)$
is maximal. Waiter then picks any vertices $a_3,a_4 \in A$ 
such that $a_3u$ and $a_4u$ are free and such that $d_W(a_3,f(S'))=d_W(a_4,f(S'))=0$, and offers these two edges to Client. 
By symmetry we may assume that Client chooses $a_3u$. Then Waiter updates $A$, $S$ and $f$ by removing $a_3$ from $A$, adding $z$ to $S$ and setting $f(z):=a_3$.
\end{enumerate}    

\medskip
  
Afterwards, Waiter repeats Stage II.  
     
\medskip

\textbf{Stage III:} 
When Waiter enters Stage III, Client's graph is a copy of
the subgraph $T'$.
Within $e(M')+1$ rounds, Waiter now forces a perfect matching between
$V(K_n)\setminus f(V(T'))$ and $f(N_T(L'))$.
The details of how Waiter can do this are given later in the strategy discussion.\\     
       
{\bf Strategy Discussion:} 
It is obvious that, if Waiter can follow the given strategy without forfeiting the game, %then 
she forces a copy of $T$ within at most 
$d_T(v)+e(T')-d_T(v)+e(M')+1=e(T)+1=n$ rounds. Hence, it remains to show that Waiter can indeed do so.
To this end, we define a vertex $u$ to be a \textit{stopping vertex} if 
$u \in f(S)$ and if for the vertex 
$t=f^{-1}(u)$ the following holds: 
for every $y \in N_{T'}(t) \setminus S$ we have $d_{T'}(y)=1$. 
We first observe the following which will help us later to show that
Waiter can follow the proposed strategy.

\begin{obs}\label{obs:tree_caseB}
Throughout Stages I and II, as long as Waiter can follow the proposed strategy, it holds that
\begin{enumerate}
\item[(i)] $|A|\geq \mu' \sqrt{n} - 2$ and $e_{C\cup W}(A)=0$,
\item[(ii)] $d_W(x,A)\leq d_C(x)$ for every $x\in f(S)$, and
\item[(iii)] $d_W(x,f(S'))\leq 1$ for every $x\in A$.
\item[(iv)] Assume that so far $\mathcal{O}_{T'}$ 
did not consist solely of a stopping vertex. Then $e_W(f(S'),A) \leq \varepsilon \sqrt{n} + 1$ holds and moreover, if
$e_W(f(S'),A) = \varepsilon \sqrt{n} + 1$ holds at the end of any round,
then at the end of the following round it holds that
$e_W(f(S'),A) \leq \varepsilon \sqrt{n}$.
\end{enumerate}

\end{obs}

\begin{proof}
Property~(i) is proven analogously to Property~(i) from Observation~\ref{obs:tree_caseA}.\\
For Property~(ii), observe that $d_W(x,A)$ may only increase, when
$x\in f(S)$ (since $e_{C\cup W}(A)=0$) and Waiter offers an edge
between $x$ and a vertex $a\in A$. If the latter happens in 
Case~1 of Stage~II
(with $x$ being one of the vertices $u_1$ and $u_2$), then the vertex
$a$ is removed from $A$ by the update for that case and hence, $d_W(x,A)$
does not increase at all. Otherwise, in Stage I or in Case~2 or Case~3
of Stage~II, Waiter actually offers two edges between $x$ and $A$, 
which makes $d_C(x)$ increase by 1 at the same time. 
Hence, $d_W(x,A)$ cannot become larger than $d_C(x)$.

\medskip

Let us consider (iii) then.
In Stage I, Waiter does not claim any edges between
$A$ and $f(S')$, since $v\notin N_T(L')$. Hence,
$d_W(x,f(S'))=0$ holds for every $x\in A$ at the end of Stage I.
We now proceed by induction, looking at any round in Stage II.
In Case~1, Waiter w.l.o.g.~gets the edge 
$au_2$ 
with $u_2\in f(S)$ and $a\in A$, but then $a$ is moved from the set $A$ 
to the set $f(S)$ (and hence maybe to the set $f(S')$) 
by the update of that case. However, since
$e_W(a,A)=0$ holds by Property~(i) for the previous round, we
conclude that $d_W(x,f(S'))$ stays unchanged for every $x$ 
which remains in the set $A$.
In Case~2 or Case~3 of Stage~II, Waiter w.l.o.g.~gets the edge $ua_2$ or $ua_4$
with $u\in f(S)$, Client gets the edge
$ua_1$ or $ua_3$, while $a_i\in A$ for every $i\in [4]$, 
and then $a_1$ or $a_3$ is moved from the set $A$ to the set
$f(S)$ (and hence maybe to the set $f(S')$) by the update of that case. 
But then, in Case~2, since $e_W(a_1,A)=0$ holds by Property~(i) for the previous round and since $u\notin f(S')$ by assumption of that case, 
we conclude that $d_W(x,f(S'))$ does not increase for any $x$ 
which remains in the set $A$. Moreover, in Case~3,
since $e_W(a_3,A)=0$ holds analogously and $d_W(a_4,f(S'))=0$
was true at the end of the previous round (by the choice of $a_4$ in that case),
we conclude that after following Case~3 we have $d_W(a_4,f(S'))=1$ and
$d_W(x,f(S'))$ does not increase for any $x\neq a_4$ 
which remains in the set $A$. Hence, in either case, the degrees
$d_W(x,f(S'))$ never exceed $1$ for the vertices $x\in A$.

\medskip

It remains to verify Property~(iv). By the discussion above for Property~(iii)
we see that $e_W(f(S'),A)$ can only increase in Case~3 of Stage~II,
and if it does then 
it increases by exactly~1. Hence, it is enough to show that
if $e_W(f(S'),A)=\varepsilon \sqrt{n} + 1$ holds at the end of any round $r$
and if at this moment $\mathcal{O}_{T'}$ still does not solely consist 
of a stopping vertex,
then $e_W(f(S'),A)\leq \varepsilon \sqrt{n}$ will hold
at the end of round $r+1$.
So, assume the mentioned conditions hold. 
Then, round $r$ was played according to Case~3 of Stage~II.
That is, at the beginning of round $r$ there was only one vertex $u\in f(S)$
which was open w.r.t. $T'$, and moreover $u\in f(S')$. By assumption
the vertex $u$ was not a stopping vertex. That is, for $t=f^{-1}(u)$
we could find a vertex $y\in N_{T'}(t)\setminus S$ such that $d_{T'}(y)\geq 2$.
In round $r$, Waiter played according to Case~3 of Stage~II, thus
the vertex $z$ (for the strategy described in Case~3) with $tz\in E(T')$ 
was chosen such that $d_{T'}(z)\geq 2$.
Waiter then fixed vertices $a_3,a_4 \in A$ 
such that $a_3u$ and $a_4u$ were free and %then 
she offered these two edges to Client. 
By symmetry, we may assume that Client chose $a_3u$ and the other edge $a_4u$
was added to Waiter's graph. Then Waiter updated $A$, $S$ and $f$ by removing $a_3$ from $A$, adding $z$ to $S$ and setting $f(z)=a_3$. Using that $d_{T'}(z)\geq 2$, we conclude that $a_3$
needs to be open w.r.t. $T'$ at the end of round $r$.
Moreover, using Property~(i),
we also get that $d_W(a_3,A)=0$ holds at this moment.
For round $r+1$, two possible cases %then 
may occur now.
The first case is that $u$ is still an open vertex w.r.t. $T'$
at the beginning of round $r+1$. Then this round is played according to Case~1
with $\{u_1,u_2\}=\{u,a_3\}$. If Waiter can follow the strategy, %then 
we already know from the discussion of Property~(iii) that
$d_W(x,f(S'))$ stays unchanged for every $x$ 
which remains in the set $A$. However, by the strategy of Case~1, 
it also happens that Waiter picks some vertex $a\in A$ with $d_W(a,f(S'))\geq 1$
such that $u_1a$ and $u_2a$ are free. Note that such a vertex $a$ exists
since by Property~(iii) and under assumption of $e_W(f(S'),A)=\varepsilon \sqrt{n} + 1$ there exist $\varepsilon \sqrt{n} + 1$ vertices $a\in A$
with $d_W(a,f(S'))\geq 1$, while 
$$
d_W(u,A) + d_W(a_3,A) \stackrel{\text{(ii)}}{\leq} 
d_C(u) + 0 \leq d_T(f^{-1}(u)) \leq \varepsilon \sqrt{n}~ .
$$
At the end of round $r+1$ the vertex $a$ %then
gets removed from $A$, and hence $e_W(A,f(S'))$ gets reduced by 
$d_W(a,f(S'))\geq 1$.\\
The second case is that $u$ is not an open vertex w.r.t. $T'$
at the beginning of round $r+1$, and hence $a_3$ is the only
open vertex w.r.t. $T'$ at that point. Since $t=f^{-1}(u)\in N_T(L')$ holds
and $tz\in E(T')$ holds for the vertex $z=f^{-1}(a_3)$
we know that $z\notin N_T(L')$ by the choice of $M'$ and $L'$.
Hence $a_3\notin f(S')$ and thus, in round~$r+1$, Waiter plays according to Case~2 (with $u:=a_3$). 
That is, Waiter then offers two edges $a_1a_3$ and $a_2a_3$
such that $a_1,a_2 \in A$ and such that 
$d_W(a_i,f(S'))\geq 1$ holds for $i\in [2]$ 
(which is possible since $d_W(a_3,A)=0$ and since
there exist $\varepsilon \sqrt{n} + 1$ vertices $a\in A$
with $d_W(a,f(S'))\geq 1$).
By symmetry we may assume that Client chooses $a_1a_3$.
Then $a_1$ gets removed from $A$ by the update of that case,
making sure that  $e_W(A,f(S'))$ gets reduced by $d_W(a_1,f(S'))\geq 1$. 
\end{proof}

\medskip

With Observation~\ref{obs:tree_caseB} in hand, we can check easily that
Waiter can follow the proposed strategy without forfeiting the game.

\medskip

{\bf Stage I:} According to the strategy, Waiter needs to offer $2d_T(v) \leq \frac{2n}{3}$ edges at $v$. She can easily do so, since there exists $n-1$ edges to choose from.

\medskip
 
{\bf Stage II:} Assume Waiter needs to make a move according to Stage~II,
but she could follow her strategy in all the previous rounds.
Further, let us assume first that $\mathcal{O}_{T'}$ does not solely consist
of a stopping vertex yet.
In Case~1, when there exists $u_1,u_2\in \mathcal{O}_{T'}$,
we %then 
know that
$$
d_W(u_1,A)+d_W(u_2,A) \stackrel{\text{(ii)}}{\leq} 
d_C(u_1) + d_C(u_2)
\leq d_{T'}(f^{-1}(u_1)) + d_{T'}(f^{-1}(u_2)) \leq 2\varepsilon \sqrt{n} < 
|A|~ ,
$$
where the last inequality uses Property~(i) and $\varepsilon < \frac{\mu}{20} < \frac{\mu'}{6}$.
Hence, Waiter can find a vertex $a\in A$ such that $u_1a$ and $u_2a$ are free edges, and hence she can follow her strategy in that case.
In Case~2 and Case~3, when there exists a unique vertex $u\in \mathcal{O}_{T'}$
which is not a stopping vertex,
we similarly obtain that $d_W(u,A)+e_W(A,f(S'))<|A|-2$
by Properties~(i)--(iv), and hence
Waiter can find vertices $a_1,a_2$ or $a_3,a_4$ as required to follow the strategy.

Now let us assume that at some point $\mathcal{O}_{T'}$ solely exists of a stopping vertex $u$. Then, in order to finish with Stage~II, only 
the vertex $u$ needs to get closed w.r.t. $T'$. As this takes at most 
$d_{T'}(f^{-1}(u)) \leq \varepsilon \sqrt{n}$ rounds played according to Case~2 or Case~3, while the Properties~(ii) and (iv) were true before that point,
we know that until the end of Stage~II, $e_W(f(S'),A)$ and $d_W(u,A)$ cannot exceed $2\varepsilon \sqrt{n} + 1$. But then, observing analogously that 
$$d_W(u,A) + e_W(f(S'),A) < 2(2\varepsilon \sqrt{n} + 1)<|A|-2~ ,$$
it follows that Waiter can find vertices $a_1,a_2$ or $a_3,a_4$ as desired
by her strategy.

\medskip

{\bf Stage III:}
When Waiter enters Stage III, Client's graph is a copy of
the subgraph $T'$. The sets
$A:=V(K_n)\setminus f(V(T'))$ and $B:=f(S')=f(N_T(L'))$ both
have size at least $\mu' \sqrt{n}=e(M')$. Moreover
$e_W(A,B)<2\varepsilon \sqrt{n}$ holds, as explained in the discussion of Stage~II, and $e_C(A,B)=0$. Following the strategy given for Theorem~\ref{WCUPM}, Waiter can force a perfect matching between $A$ and $B$
within $e(M')+1$ rounds. 
\end{proof}

\medskip

Having Theorem~\ref{thm:tree_stronger} in hands, we are now able to prove Theorem~\ref{thm:tree} and Theorem~\ref{thm:tree_factor}.\\

\textit{Proof of Theorem~\ref{thm:tree}.}
Let $n$ be large enough, and let $T$ be any tree on $n$ vertices
and with maximum degree at most $\varepsilon \sqrt{n}$.
The upper bound $\tau_{WC}(\mathcal{F}_T,1)\leq n$
follows from Theorem~\ref{thm:tree_stronger};
the lower bound $\tau_{WC}(\mathcal{F}_T,1)\geq n-1$ 
trivially holds since $e(T)=n-1$.\\
If $T$ is a path on $n$ vertices, then 
$\tau_{WC}(\mathcal{F}_T,1)= n-1$. 
Indeed, in the strategy given for Theorem~\ref{WCUHC},
Waiter forces a Hamilton path in the first round of Stage~III,
which is the $(n-1)^{\text{st}}$ round in the game.
This shows that the lower bound in Theorem~\ref{thm:tree} is tight.\\
If $T$ is a tree obtained from a path on $n-4$ vertices
by connecting two further vertices to each of its endpoints,
then $\tau_{WC}(\mathcal{F}_T,1) = n$. 
Indeed, if Waiter would want to force a copy of $T$
within $n$ rounds, 
then for some edge $e\in E(T)$ she would need 
to force a copy of $T-e$ within $n-2$ rounds. 
However, since a unique edge extends this copy of $T-e$ 
to a copy of $T$, 
Client can easily prevent the copy of $T$ 
in round $n-1$, because Waiter needs to offer two edges. This shows that the upper bound is tight.
\hfill $\Box$

\medskip\medskip

\textit{Proof of Theorem~\ref{thm:tree_factor}.}
Let $T$ be any tree on $k$ vertices, and let $H$ by a $T$-factor on $n$ vertices with $k|n$. 
Then $\tau_{WC}(\mathcal{F}_{n,T-fac},1)\geq e(H)=\frac{k-1}{k}n$.
In order to prove the upper bound on 
$\tau_{WC}(\mathcal{F}_{n,T-fac},1)$, let $H'$ by an arbitrary tree on $n'=n+1$ vertices, obtained from $H$ by adding one further vertex $v'$ and adding exactly one edge between $v'$ and each copy of $T$ in $H$. Waiter then pretends to play on 
$K_{n'}\supset K_n$ with $V(K_{n'})\setminus V(K_n)=\{p'\}$.
She plays according to the strategy given for Theorem~\ref{thm:tree_stronger} (with $v:=v'$ and $p:=p'$), and whenever this strategy makes her offer two edges 
incident to $p'$, she only pretends to play that round.
This way, she forces a copy of $H=H'-v'$ in the game on $K_n$,
wasting at most one round, and hence she wins within
$\frac{k-1}{k}n+1$ rounds.\\
For the tightness of both bounds, we can
use the same trees as in the proof of Theorem~\ref{thm:tree}. For large enough $k$, if $T$ is a path on $k$ vertices and $H$ is a $T$-factor, Waiter can win within $\frac{k-1}{k}n$. Before the game starts, she just splits the vertex set into $\frac{n}{k}$ sets of size $k$,
and then on each of these parts she forces a Hamilton path without wasting a move. On the other hand, if $T$ is a 
tree obtained from a path on $k-4$ vertices
by connecting two further vertices to each of its endpoints,
then analogously to the proof of Theorem~\ref{thm:tree} we get
$\tau_{WC}(\mathcal{F}_{n,T-fac},1)= \allowbreak \frac{k-1}{k}n+1$. 

\hfill $\Box$

% Unbiased Triangle factor game
\section{Unbiased triangle factor game}\label{sec:triangles}
In this section we give the proof of Theorem~\ref{thm:triangle_factor}. However, before doing so, let us first prove the following lemma.

\begin{lemma}
	\label{TFTwoTriangles}
Consider an unbiased Waiter-Client game on $K_{12}$, and fix any two vertices $u,v \in V(K_{12})$. Then Waiter has a strategy that, within $7$ moves, forces Client to create $2$ vertex disjoint triangles with the following additional properties:
\begin{enumerate}
	\item Both $u$ and $v$ are in the two triangles.
	\item All edges within the set of $6$ vertices, that are not in a triangle, and the edge $uv$ have not been offered.
\end{enumerate}
\end{lemma}

\begin{proof}
	Throughout the proof we will often use the fact, that Waiter can offer two edges from the same vertex to two new vertices. Then, by symmetry, Client's choice does not affect the rest of Waiter's strategy.	
	First Waiter forces Client to create two vertex disjoint cherries rooted at $u$ and $v$ respectively, which she can do by the above remark
	and since we have 12 vertices in total. Note that this leaves four other vertices connected to either $u$ or $v$ by an edge of Waiter. Suppose Client's graph now has the edge set $\{ub_1,ub_2,vb_3,vb_4\}$.
	Next, Waiter offers $b_1b_2$ and $b_3b_4$, forcing Client to close a triangle. Without loss of generality, we may assume that Client chooses $b_1b_2$, closing a triangle containing $u$. Afterwards, Waiter forces Client to take an edge incident to $v$, say $vw$, by offering two edges from $v$ to new vertices, and finally she offers the edges $wb_3$ and $wb_4$, forcing Client to close a second triangle which this time contains $v$. Note that
	both of the claimed properties are satisfied.
\end{proof}

\textit{Proof of Theorem~\ref{thm:triangle_factor}.}

For the lower bound on $\tau_{WC}(\mathcal{F}_{n,K_3-fac},1)$
we may provide a {\bf Client's strategy}. 
Throughout the game Client will maintain a set of {\em marked} vertices $M \subset V$, which is initially empty.
Moreover, we consider the sets $X_v:=\{xy:xv,vy\in E(C)\}$ for every $v\in V(K_n)$ and the set $X = \bigcup_v X_v$ which consists of those edges which would close a triangle in Client's graph. In the following we describe Client's strategy.

In any round of the game suppose that Waiter offers two edges $x_1y_1$ and $x_2y_2$ to Client. Then Client chooses his edge according to the following case distinction.

\begin{enumerate}
	\item[] {\bf Case 1:} Suppose at least one of the offered edges belongs to $E(K_n)\setminus X$. Then Client arbitrarily chooses such an edge.
	\item[] {\bf Case 2:} Suppose otherwise that there exist $z_1,z_2\in V(K_n)$ such that $x_1y_1 \in X_{z_1}$ and 
	$x_2y_2 \in X_{z_2}$. Client then considers three subcases:
	\begin{enumerate}[label=(\alph*)]
		\item If $z_1 \not \in M$, Client chooses the edge $x_2 y_2$ and adds $z_1$ to the set $M$.
		\item Otherwise, if $z_2 \not \in M$, Client chooses the edge $x_1 y_1$ and adds $z_2$ to the set $M$. 
		\item Otherwise, if $z_1,z_2 \in M$, then Client chooses his edge arbitrarily. 
	\end{enumerate}
\end{enumerate}

It is obvious that Client can always follow the proposed strategy. Hence, it remains to show that it prevents a triangle factor for at least $\frac{13}{12}n$ rounds.
We start with the following observation.

\begin{obs}
	\label{TFLemma}
	At the end of the game, let 
	$T=\{t_1,\ldots,t_{\frac{n}{3}}\}$ 
	be a triangle factor in $C$ 
	and let $U=\{t \in T : t \cap M = \varnothing\}$ 
	be the set of triangles, for which all three vertices are 
	not marked. Then the following properties hold:
	\begin{enumerate}[label=(\alph*)]
		\item $\forall m \in M$ we have $d_C(m) \ge 3$. \label{TFMHighDegree}
		\item $|M| \ge |U|$. \label{TFMIsLarge}
	\end{enumerate}
\end{obs}

\begin{proof}
To show \ref{TFMHighDegree}, fix a vertex $m \in M$ and consider the turn in which Client added $m$ to $M$. 
Then, according to Case 2 of the given strategy,	
at that point of time Client was offered an edge $xy$ with $xm,my \in E(C)$, but did not add this edge to his graph, i.e. the triple $\{m,x,y\}$ does not form a triangle in $C$. However, since $m$ needs to be in a triangle, $d_{C}(m) \ge 3$ follows.
	For \ref{TFMIsLarge}, fix a triangle $t \in U$ and consider the turn in which Client completed this triangle. Then Client must have played according to Case 2(a) or Case 2(b), i.e. a case in which Client adds a vertex to $M$. 
	Hence, $|M| \ge |U|$ follows immediately.	
\end{proof}

\medskip

Suppose that a triangle factor $T=\{t_1,\ldots,t_{\frac{n}{3}}\}$ is created in Client's graph.
In order to conclude that at least $\frac{13}{12}n$ rounds have been played, we consider two cases.
Assume first that right at this moment $|U| \ge \frac{n}{6}$.
Then 
\begin{align*}
2|E(\mathcal{C})| & = \sum_{v \in V} d_{C}(v)
		= \sum_{v \in M} d_{C}(v)+ \sum_{v \in V\setminus M} d_{C}(v)\\
		& \ge 3\cdot|M| + 2\cdot|V\setminus M|
		= |M| + 2\cdot|V|
		\ge \frac{n}{6}+2n
		= \frac{13n}{6}
\end{align*}
where the first inequality follows from Observation~\ref{TFLemma}\ref{TFMHighDegree} and since every vertex belongs to a triangle, and the second inequality
follows from Observation~\ref{TFLemma}\ref{TFMIsLarge}.
Assume $|U| < \frac{n}{6}$ next, then
\begin{align*}
2|E(\mathcal{C})| & = \sum_{v \in V} d_{C}(v)
		= \sum_{t \in T} \sum_{v\in t} d_{C}(v)
		 \ge  6\cdot|U| + 7\cdot(\frac{n}{3}-|U|)
		 > \frac{7n}{3}-\frac{n}{6}
		 = \frac{13n}{6}
		\end{align*}
where the first inequality follows from Observation~\ref{TFLemma}\ref{TFMHighDegree}. In any case, we obtain
$|E(C)| \geq \frac{13}{12}n$.\\

For the upper bound on $\tau_{WC}(\mathcal{F}_{n,K_3-fac},1)$ we may provide a {\bf Waiter's strategy}. 
For this, let $n_0$ be an even integer such that Waiter has a strategy to force a copy of $K_{48}$ in the unbiased Waiter-Client game on $K_{n_0}$. Such an integer exists by \cite{Beckbook}. Now, playing on $K_n$,
fix any set of vertices $W\subset V(K_n)$ with $|W|=n_0$.
In the following we describe Waiter's strategy.
If at any point during the game, Waiter is unable to follow the strategy, she forfeits the game. (We will later see that this does not happen.) Waiter's strategy consists of the following three stages: 

\medskip
	
	\textbf{Stage I:} Playing on $K_n[W]$ only, Waiter 
	forces Client to create a clique of size $48$.
	
	\medskip
	
	\textbf{Stage II:} When Waiter enters Stage~II,
	there exists a set $K\subset W$ of size $48$
	such that $C[K] \cong K_{48}$.
	Let $S = W\setminus K$ and $T = V \setminus W$. 
	Waiter will force Client to create a large family
	of vertex disjoint triangles and she will update 
	$S$ and $T$ by always removing the vertices
	of these triangles. As long as $|T|\geq 12$,
	she plays in sequences of at most $7$ moves as follows:
	
	Waiter first arbitrarily chooses $12$ vertices in 
	$S\cup T$, where she chooses exactly $2$ vertices
	from $S$ if $S\neq \varnothing$. Among these 12 vertices	
	she fixes two vertices $u,v$, with
	$u,v\in S$ if $S\neq \varnothing$, 
	and then she plays according to the strategy
	from Lemma~\ref{TFTwoTriangles} 
	on these 12 vertices.  By this, two triangles
	are forced in Client's graph. Waiter then
	removes the vertices
	which belong to these triangles from $S$ and $T$,
	respectively.

	\medskip

	\textbf{Stage III:} When Waiter enters Stage~III,
	we have $|T|<12$. Then for each of the vertices $v\in T$ 
	Waiter picks pairwise disjoint sets $K_v$ 
	of four vertices in $K$ 
	and offers the edges from $E_{K_n}(v,K_v)$ in pairs.

\medskip

It is clear that, if Waiter can follow the proposed strategy without forfeiting the game, %then 
she forces a
triangle factor within at most
$
\binom{n_0}{2} + \frac{7}{6}n + 2\cdot 12 = \frac{7}{6}n + O(1)
$
rounds. Therefore, it remains to verify that Waiter indeed can follow the proposed strategy.\\

{\bf Strategy discussion:}
The strategy in Stage~I can be followed by the choice of $n_0$.
Note that when Stage~I is over, no edges from 
$E_{K_n}(T, S \cup T)$ have been offered yet. 

For Stage~II assume that, before Waiter plays a sequence of moves
as described in the strategy, it is still true that
all edges in $E_{K_n}(T, S \cup T)$ are free. It then follows
from Lemma~\ref{TFTwoTriangles} that Waiter can play her next moves
according to the strategy of Stage~II. Just note that %, 
in the case when $S\neq \varnothing$ it may happen that $uv\in E_{K_n}(S)$ has already been offered before; but this does not cause any problem, since
for Lemma~\ref{TFTwoTriangles} Waiter does not need to offer $uv$ at all.
Finally also note that, when the two triangles are created,
%then 
by (2) from Lemma~\ref{TFTwoTriangles}
and by the update in Stage~II it follows that $E_{K_n}(T, S \cup T)$ again consists only of free edges.
Hence, Waiter can repeatedly apply Lemma~\ref{TFTwoTriangles}
and follow the strategy for Stage~II.

Afterwards, when Waiter enters Stage~III, it holds that $S=\varnothing$ and $|T|<12$. Since $E_{K_n}(K,T)$
consists solely of free edges at this point, Waiter can offer edges as
desired.  \hfill $\Box$

%Biased games
\section{Biased games}

In the proof of Theorem~\ref{thm:pm_biased} and Theorem~\ref{thm:ham_biased} we will make use of the following result due to 
Bednarska-Bzd\c{e}ga, Hefetz, Krivelevich and \L uczak~\cite{BHKL2016}.

\begin{thm}[Theorem 1.4(ii) in~\cite{BHKL2016}]\label{thm:pancyclic}
There exists a positive constant $c\in (0,1)$ and an integer $n_0$ such that the following holds. If $n\geq n_0$ 
and $b \leq cn$, then playing a $b$-biased Waiter-Client 
game on $E(K_n)$, Waiter has a strategy to force
a spanning pancyclic graph.
\end{thm}

We start with the proof of Theorem~\ref{thm:ham_biased}.

\begin{proof}[Proof of Theorem~\ref{thm:ham_biased}]
Let $c$ and $n_0$ be given according to Theorem~\ref{thm:pancyclic}.
We set 
\begin{equation}\label{eq:biased_constants}
C_0 = 100\max\{c^{-1},n_0\}, ~~~
\delta_0 = 0.1c,~~~ 
\delta = 0.01\min\{\delta_0, C_0^{-1}\},~~~
C=C_0\delta^{-1}.
\end{equation}

We let $b\leq \delta n$ from now on and, whenever necessary, we will assume $n$ to be large enough.

In the following we will describe a strategy for Waiter in the $b$-biased Waiter-Client game on $K_n$, and afterwards we will show that it is a strategy with which Waiter forces a Hamilton cycle within at most $n+Cb$ rounds. Whenever Waiter is not able to follow the proposed strategy, she forfeits the game. (We will show later that this does not happen.)
The strategy is split into five stages.

\medskip

{\bf Stage I:} Within $n-C_0b-1$ rounds, Waiter forces a path $P=(a_1,\ldots,a_{n-C_0b})$ on $n-C_0b$ vertices according to the following rule:

Initially set $P=\{a_1\}$ for an arbitrary vertex $a_1\in V(K_n)$. Assume after $i-1$ rounds Waiter has already forced 
a path $P=(a_1,\ldots,a_i)$ on $i$ vertices. Then in round $i$ Waiter selects $b+1$ vertices $x_1,\ldots,x_{b+1}$ from $V\setminus V(P)$ which have the smallest degree in her graph. Then she offers the edges 
\{$a_{i}x_j:j \in [b+1] \}$ to Client, of which he needs to pick one. % of the edges $a_{i}x_j$
Waiter updates the path such that $P=(a_1,\ldots,a_{i+1})$ with $a_{i+1}:=x_j$.

Once $P$ has reached length $n-C_0b-1$, Waiter proceeds with Stage~II.

\medskip

{\bf Stage II:} Let $P$ denote the path that Client claims at the end of Stage~I. Let $R=V\setminus V(P)$ then.
Playing only on $K_n[R]$, within at most $Cb$ rounds, Waiter forces a Hamilton cycle of $K_n[R]$. The details of how she can do this, can be found later in the strategy discussion. Afterwards, Waiter proceeds with Stage~III.

\medskip

{\bf Stage III:} This stage lasts 1 round. Let $P=(a_1,\ldots,a_{n-C_0b})$ be the path from Stage I and let $H$ be the Hamilton cycle from Stage~II. 
Waiter now picks $b+1$ vertices $x_1,\ldots,x_{b+1}\in V(H)$ such that $a_1x_j$ is free for every $j\in [b+1]$. Then she offers all these $b+1$ edges.

Afterwards, Client needs to pick one of the edges $a_{1}x_j$.
From then on, set $\tilde{x}=x_j$ and let 
$x$ be one of the neighbours of $x_j$ on $H$. 
Next, Waiter proceeds with Stage~IV.

\medskip

{\bf Stage IV:} This stage lasts exactly $b$ rounds in which Waiter forces a few Hamilton paths on $V(K_n)\setminus R$
using P\'osa rotations~\cite{Posa1976}.

More precisely, let $P_0=P=(a_1,\ldots,a_{n-C_0b})$ be the path from Stage~I. Set $v_0=a_{n-C_0b}$. By playing on $K_n[V(P_0)]$ only, for $i\in [b]$ Waiter will ensure that immediately after the $i^{\text{th}}$ round in Stage~IV
Client's graph contains a path $P_i$ such that the following properties hold:

\medskip

\begin{minipage}{0.95\textwidth}
\begin{enumerate}[label=\itmarab{P}]
\item\label{paths:vertices} $V(P_i)=V(P_0)$,
\item\label{paths:endpoints} $P_i$ has endpoints $a_1$ and $v_i\in V\setminus \{v_0,\ldots,v_{i-1}\}$,
\item\label{paths:threat} $v_ix$ is free.
\end{enumerate}
\end{minipage}

\medskip

The details of how Waiter can do this can be found later in the strategy discussion. Afterwards, Waiter proceeds with Stage~V.

\medskip

{\bf Stage V:} Within one round, Waiter forces Client's graph to contain a Hamilton cycle of $K_n$.
The details of how she can do this can be found later in the strategy discussion.

\medskip

If Waiter can follow the proposed strategy without forfeiting the game, then it is obvious that she forces a Hamilton cycle within at most
$
(n-C_0b-1) + Cb + 1 + b + 1 < n + Cb
$
rounds.
Therefore, it remains to prove that Waiter indeed can always
follow the proposed strategy. \\

{\bf Strategy discussion:}

{\bf Stage I:} Consider the round $i\in [n-C_0b-1]$ in Stage~I. By then 
Waiter already forced a path $P=(a_1,\ldots,a_i)$.
Since in the previous rounds Waiter only offered edges which are incident to at least one of the vertices $a_j$ with $j<i$, we know that before Waiter's $i^{\text{th}}$ turn
all the edges between $a_i$ and $V\setminus V(P)$ are free. Moreover, since $|V\setminus V(P)|\geq C_0b$,
Waiter can easily find and offer $b+1$ edges as required by the strategy.

\medskip

{\bf Stage II:} Let $P=(a_1,\ldots,a_{n-C_0b})$ denote the path which Waiter has forced at the end of Stage~I. Since so far she only offered edges which are incident to at least one of the vertices $a_j$ with $j<n-C_0b$, we know that at the beginning of Stage~II all the edges inside $R=V\setminus V(P)$ are still free. Let $\tilde{n}:=|R|=C_0b$ and, using (\ref{eq:biased_constants}), observe that
$b=C_0^{-1}\tilde{n}<c\tilde{n}$ and $\tilde{n}>n_0$. 

According to Theorem~\ref{thm:pancyclic}, Waiter 
has a strategy for playing a $(\tilde{b} : 1)$ game on $K_n[R]$ with bias $\tilde{b}=c\tilde{n}$
in such a way that Client is forced to get a pancyclic graph.
Thus, following that strategy with bias $b<\tilde{b}$
(by pretending to add $\tilde{b}-b$ extra edges to Waiter's graph in each round), Waiter can force a Hamilton cycle
on $R$ within at most
$$
\left\lceil \frac{\binom{|R|}{2}}{\tilde{b}+1} \right\rceil
<
\frac{|R|^2}{\tilde{b}} = 
\frac{\tilde{n}\cdot C_0b}{c \tilde{n}} = 
C_0c^{-1}b \stackrel{(\tref{eq:biased_constants})}{<} C b
$$ 
rounds, as promised in the strategy description.

\medskip

{\bf Stage~III:} When Waiter enters Stage~III, it holds that
$d_{C\cup W}(a_1)=b+1$, since only in the very first rounds
she offered edges incident to $a_1$.
The number of available edges between $a_1$ and $R$ is at least 
 %In particular,
%$$
$%d_F(a_1,R)\geq 
|R|-d_{C\cup W}(a_1) = C_0b - (b+1) > b+1$
%$$
by the choice of $C_0$. %\footnote{\textcolor{red}{need to define $F$ as free edges graph in the notation section + maybe use this notation in other sections}} 
Hence, Waiter can offer edges as required for this stage of the proposed strategy.

\medskip

{\bf Stage~IV:} Before we will show that Waiter can follow Stage~IV, let us first observe that none of the vertices has a too large degree by now.

\begin{obs}\label{obs:degrees}
At the beginning of Stage~IV it holds that
$d_{C\cup W}(v,V(P_0))<\delta_0 n$ for every vertex $v\in V(K_n)$.
\end{obs}

\begin{proof}
When Waiter forces the path $P$ in Stage~I,
she always prefers to offer edges from the current endpoint $a_i$ to the vertices of smallest Waiter-degree in $V\setminus V(P)$. This way she makes sure that the Waiter-degrees among the vertices in $V\setminus V(P)$ differ by at most 1
throughout Stage~I. Now, Stage~I lasts $n-Cb-1$ rounds, and thus 
$e(W)\leq n(b+1)$
holds throughout Stage~I. In particular, all vertices
$v\in V\setminus V(P)$ then satisfy
$$
d_{C\cup W}(v)< \left\lceil \frac{2e(W)}{|V\setminus V(P)|} \right\rceil < \frac{2n(b+1)}{C_0b} \stackrel{(\tref{eq:biased_constants})}{<} 0.5\delta_0 n~ .
$$

Now consider the beginning of Stage~IV. 
It holds that $d_{C\cup W}(a_1)=2(b+2)<3\delta n<0.1\delta _0 n$, since there were only two rounds in which Waiter offered edges at $a_1$.
For every vertex $v\in V(P)\setminus \{a_1\}$ we then have
$d_{C\cup W}(v) < 0.5\delta_0 n + (b+1) < \delta_0 n$,
since after $v$ was added to $P$ there was only
one round in which Waiter offered edges incident to $v$.
Moreover, for every remaining vertex $v$ (i.e. $v\in R$),
we have $d_{C\cup W}(v,V(P_0))<0.5\delta_0 n + 1 < \delta_0 n$ 
since these vertices belong to $V\setminus V(P)$ at the end of Stage~I and since afterwards, until the end of Stage~III, the edge $va_1$ may be the only edge between $v$ and $V(P)$ that got offered. This proves the observation.
\end{proof}

\medskip

Having the observation in hands, we now show how Waiter can force the desired paths in Stage~IV. We will do it in such a way that the Properties~\ref{paths:vertices}--\ref{paths:threat} hold as well as the following property:

\medskip

\begin{minipage}{0.95\textwidth}
\begin{enumerate}[label=\itmarab{Q}]
\item\label{paths:degree1} $d_{C\cup W}(v,V(P_i))<\delta_0 n + i$ for every $v\in V(P_i)\setminus \{v_0,\ldots,v_{i-1}\}$.
\end{enumerate}
\end{minipage}

\medskip

We proceed by induction on $i$. For $i=0$ the path $P_0=P$ from
Stage~I trivially satisfies \ref{paths:vertices} and \ref{paths:endpoints}. Property~\ref{paths:degree1} follows by Observation~\ref{obs:degrees}. Moreover, Property~\ref{paths:threat}
holds by the following reason: In Stage~I every offered edge is incident to at least one of the vertices in $V(P_0)\setminus \{v_0\}$, in Stage~II every edge is disjoint from $V(P_0)$ and in Stage~III every edge is incident to $a_1\notin \{v_0,x\}$.
Thus, $v_0x$ has not been offered yet.

So, let $i>0$ then. Let $P_{i-1}$ be the path given by induction and consider the $i^{\text{th}}$ move in Stage~IV.
For every vertex 
$v\in V(P_{i-1})\setminus \{v_{i-1}\}$ denote with 
$v^+$ the unique neighbour of $v$ in $P_{i-1}$ with
$\dist_{P_{i-1}}(v^{+},v_{i-1})=\dist_{P_{i-1}}(v,v_{i-1})-1$.
Set
\begin{align*}
B_1 & :=
	\left\{
	y\in V(P_{i-1}):~
	y^{+}x\in C\cup W
	\right\} ~ , \\
B_2 & :=
	\left\{
	y\in V(P_{i-1}):~
	y^{+}=v_j~ \text{for some }j<i
		%\text{ or }
		%y^{+}=a_1 \textcolor{red}{why possible?}
	\right\} ~ , \\
B_3 & :=
	\left\{
	y\in V(P_{i-1}):~
	yv_{i-1}\in C\cup W
	\right\} ~ .
\end{align*}
By Observation~\ref{obs:degrees} and since
in Stage~IV Waiter only offers edges in $V(P_0)$,
we obtain $|B_1|<\delta_0 n$. Since $i\leq b$, we have
$|B_2|\leq i\leq b < \delta_0 n$. Using Property~\ref{paths:degree1} we get 
$|B_3|<\delta_0 n + i \leq 2\delta_0 n$.
Hence, also using (\ref{eq:biased_constants}), we conclude
\begin{align*}
|V(P_0)\setminus (B_1\cup B_2\cup B_3)|
& \geq
n- C_0b - 4\delta_0 n \geq
n- C_0\delta n - 4\delta_0 n 
\geq n - 0.01n - 0.4 n \\
& \geq b+1
\end{align*}
provided $n$ is large enough. Waiter's strategy now is to offer
$b+1$ edges of the form $yv_{i-1}$ with 
$y\in V(P_0)\setminus (B_1\cup B_2\cup B_3)$, which is possible since $y\notin B_3$ implies that $yv_{i-1}$ is free.
Client then needs to claim one of these edges; by abuse of notation denote this edge with 
$yv_{i-1}$.
Let $v_i=y^+$ then, and let $P_i$ be the path
induced by $(E(P_{i-1)})\cup \{yv_{i-1}\}) \setminus \{yy^+\}$.
Property~\ref{paths:vertices} is trivially satisfied.
Moreover, $P_i$ has endpoints $a_1$ and $y^+=v_i$ and, since $y\notin B_2$, Property~\ref{paths:endpoints} holds as well. 
Property~\ref{paths:threat} is guaranteed as $y\notin B_3$.
For Property~\ref{paths:degree1} observe the following: 
In her $i^{\text{th}}$ move of Stage~IV, 
Waiter only offers edges incident to $v_{i-1}$.
Thus, the degrees 
$d_{C\cup W}(v,V(P_{i-1}))
	=d_{C\cup W}(v,V(P_{i}))$ can increase at most by $1$
	for every $v\neq v_{i-1}$.

\medskip

{\bf Stage V:} Let $v_0,\ldots,v_b$ be the distinct endpoints
of the paths $P_0,\ldots,P_b$ from Stage~IV. Each of the $b+1$ edges $v_ix$ with $0\leq i\leq b$ is free, and hence Waiter can offer those in her next move. Let $v_jx$ be the edge
claimed by Client afterwards, and let $f=a_1\tilde{x}$ be the edge which Client claimed in Stage~III. Then 
$$(E(P_j)\cup E(H) \cup \{f,v_jx\})\setminus \{x\tilde{x}\}$$ is a Hamilton cycle of $K_n$ which is fully claimed by Client.
\end{proof}

Finally, we proceed with the proof of Theorem~\ref{thm:pm_biased}. 

\begin{proof}[Proof of Theorem~\ref{thm:pm_biased}]
Creating a perfect matching, under assumption of Theorem~\ref{thm:pancyclic}, is rather straightforward. Let $\delta_0,\delta,C_0$ and $C$
be defined as in the previous proof. 
For roughly $0.5(n-C_0b)$ rounds, Waiter can force a large matching, by playing according to Stage~I from the previous proof, where a long path was created, and just faking every second move (i.e. pretending to make the move, but not playing at all). Then, as in Stage~II of the previous proof,
she forces a Hamilton cycle on the remaining vertices within $Cb$ rounds, thus a perfect matching is created. 
\end{proof}\label{sec:biased}

%Concluding remarks
\section{Concluding remarks}\label{sec:concluding}
%{\color{red}todo}
{\bf Winning as fast as possible.}
As already mentioned in the introduction, there are quite a few games which Waiter win (almost) perfectly fast. Also, for almost every game that we considered in our paper, we were able to prove that Waiter can win it at least asymptotically fast. On the other hand, for the triangle factor game we know that it is not won asymptotically fast.
We believe that the upper bound in Theorem~\ref{thm:triangle_factor} is asymptotically tight
and hence pose the following conjecture.

\begin{conj}
It holds that $\tau_{WC}(\mathcal{F}_{n,K_3-fac},1)=\frac{7}{6}n+o(n)$.
\end{conj}

Another game to consider is the \emph{minimum degree $k$ game} $\mathcal{D}_{n,k}$ played on $K_n$, 
in which the winning sets consist of all spanning subgraphs $H$ with $\delta(H)\geq k$.
In this paper we considered the unbiased and biased version of the perfect matching and Hamiltonicity game, which covers the minimum degree $1$ and minimum degree $2$ games.
We also have an argument that would show that
Waiter can win the unbiased game with winning sets $\mathcal{D}_{n,k}$ within
$\frac{kn}{2}+O(1)$ rounds. We wonder whether this can be improved as follows:

\begin{problem}
For $k>2$, show that $\tau_{WC}(\mathcal{D}_{n,k},1)=\lfloor \frac{kn}{2} \rfloor + 1$. Moreover, determine
$\tau_{WC}(\mathcal{D}_{n,k},b)$ asymptotically when $b>1$.
\end{problem}
Note that Maker can win the unbiased Maker-Breaker version of the game $\mathcal{D}_{n,k}$ within at most $\lfloor \frac{kn}{2} \rfloor + 1$ rounds, as shown by Ferber and Hefetz in~\cite{FH14}, therefore we are curious to know whether this holds in the Waiter-Client setup as well.

Finally, one could look at the \emph{$k$-clique factor game} for $k\geq 3$. The triangle factor game considered in this paper covers the case $k=3$, but it is still unknown what happens for $k>3$. 

\begin{problem}
Find non-trivial upper and lower bounds for $\tau_{WC}(\mathcal{F}_{n,K_k-fac},1)$ when $k>3$.
\end{problem}

Lastly, all the games we considered in this paper are examples showing that Waiter can win (asymptotically) at least as fast as Maker. We wonder whether this is always the case.

\begin{question}
Does there exist some family $\mathcal{F}$ of winning sets such that $\tau_{MB}(\mathcal{F},1)<\tau_{WC}(\mathcal{F},1)$?
\end{question}

{\bf Further questions involving trees.}
In Section~\ref{sec:trees} we found a fast winning strategy for Waiter in the case where she wants to force a copy of a given spanning tree that fulfills some maximum degree condition. It would be interesting to see if it is possible to relax this condition when we do not intend to win fast. 

\begin{problem}
For trees $T$ on $n$ vertices,
how large can $\Delta(T)$ be such that
Waiter has a strategy in the game with winning sets $\mathcal{F}_T$ on $K_n$?
\end{problem}

Another question to think about would be the biased version of this game, when the bias depends on $n$.

\begin{problem}
Determine $\tau_{WC}(\mathcal{F}_T,b)$ asymptotically when $b>1$.
\end{problem}

% References

\end{document}